\documentclass[11pt, reqno]{amsart}
\usepackage{tikz}
\usepackage{amsmath}
\usepackage{amssymb}
\usepackage{xcolor} 
\usetikzlibrary{decorations.pathreplacing, fit}
\usepackage[textheight=615pt,textwidth=360pt]{geometry}

\theoremstyle{definition}

\theoremstyle{remark}

\usepackage{hyperref}
 \newtheorem{thm}{Theorem}[section]
 
 \newtheorem{lem}[thm]{Lemma}
 
 \theoremstyle{definition}
 
 \newtheorem{rem}[thm]{Remark}
 \numberwithin{equation}{section}
 \newcommand{\e}{\varepsilon}
 \newcommand{\R}{\mathbf{R}}

 \numberwithin{equation}{section}
\usepackage{amsmath,amssymb}
\usepackage{mathrsfs}
\def\leq{\leqslant}
\def\geq{\geqslant}
\allowdisplaybreaks[3]

\begin{document}
\title[Long-Time Existence for 2-Dimensional Quasilinear Wave Equations]{Long-Time Existence of Quasilinear Wave Equations Exterior to Star-shaped Obstacle in 
$2\mathbf{D}$}

\author{Ning-An Lai}
\address{School of Mathematical Sciences, Zhejiang Normal University, Jinhua 321004, China}
\email{ninganlai@zjnu.edu.cn(N.-A. Lai)}

\author{Cui Ren}
\address{School of Mathematical Sciences, Zhejiang Normal University, Jinhua 321004, China}
\email{rencui@zjnu.edu.cn(C. Ren)}

\author{Wei Xu}
\address{Corresponding author at: School of Mathematics and Information Sciences, Nanchang Hangkong University,
Nanchang, China}
\email{70709@nchu.edu.cn(W. Xu)}



\begin{abstract}

In this paper, we study the long-time existence result for small
data solutions of quasilinear  wave equations exterior to star-shaped
regions in two space dimensions. The key novelty is that we establish a
Morawetz type energy estimate for the perturbed inhomogeneous wave equation in the exterior domain, which yields $t^{-\frac12}$ decay inside the cone. In 
addition, two new weighted $L^2$ product estimates are established to produce $t^{-\frac12}$ decay close to the cone.
We then show that the existence lifespan $T_\e$ for the quasilinear wave equations with general quadratic nonlinearity satisfies
\begin{equation*}
\varepsilon^2T_{\varepsilon}\ln^3T_{\varepsilon}=A,
\end{equation*}
for some fixed positive constant $A$, which is almost sharp (with some logarithmic loss) comparing to the known result of the corresponding Cauchy problem.

\end{abstract}

\maketitle

\section{Introduction}

In the following, the point in $[0, T]\times \R^2\backslash \mathcal {K}$ will be denoted by $X=(x_0,x_1,x_2)=(t,x)$.
Partial derivatives will be written as $\partial_l=\partial/\partial
x_l,l=1,2$, with the abbreviations
$\partial=(\partial_0,\partial_1,\partial_2)=(\partial_t,\nabla)$.
The angular momentum operator is the vector field
$\Omega=x_1\partial_2-x_2\partial_1$. Denote $r=|x|$ and
$\Gamma=\{\partial,\Omega\}$.
\par In the present paper we consider the initial boundary value problem of the quasilinear wave equation
\begin{equation}\label{Y1}
\left \{
\begin{array}{lllll}
\Box u=Q(\partial u,\partial^2u),\quad (t,x)\in\mathbb
R_{+}\times \mathbb
R^2\backslash\mathcal {K}, \\
u(t,x)|_{\partial\mathcal {K}}=0,\\
u(t_0,x)=\e f(x),\quad \partial_tu(t_0,x)=\e g(x),\quad x\in\mathbb
R^2\backslash\mathcal {K},
\end{array} \right.
\end{equation}
where
\[
\Box=\partial_t^2-\sum_{i=1}^{2}\partial_{x_i}^2
\]
denotes the standard d'Alembertian operator, $\mathcal {K}$ is a bounded  and star-shaped domain in $\R^2$, and $\e$ is a parameter denoting the smallness of 
the initial data. We assume the nonlinear term $Q(\partial u,\partial^2u)$ in system
\eqref{Y1} is quadratic and is linear in $\partial^2u$. We can
expand
\begin{align*}
Q(\partial u,\partial^2u)=B(\partial
u)+B_{\gamma}^{\alpha\beta}\partial_{\gamma}u\partial_{\alpha}\partial_{\beta}u,
\end{align*}
where $B(\partial u)$ is a quadratic form and
$B_{\gamma}^{\alpha\beta}$ are real constants with
$B_{\gamma}^{\alpha\beta}=B_{\gamma}^{\beta\alpha}$. Throughout this
paper, we use the convention that repeated indices
$\alpha,\beta,\gamma$ are summed from $0$ to $2$ and repeated
indices $i,l$ are summed from $1$ to $2$.
\par

There are satisfactory and fruitful long time behavior results for the small date Cauchy problem of scaler nonlinear wave equations in $\R^n$, see the 
classical works \cite{Joh76, Kla80, Sha82, Kla83, Joh84, Kla85, Kla86, Chr86} and \cite{Hor85,Kov87, LiY89,Katayama,Li3,Li2,Hoshiga1,Alinhac3, Alinhac4, 
Alinhac1, Alinhac2, Hoshiga2,Hoshiga3} and references therein for higher dimensional cases $(n\ge 3)$ and for $n=2$ respectively. One can find more detailed 
introduction in the remarkable monograph \cite{LiZ17}.

The first major developments from the results for Cauchy problem mentioned above to the corresponding exterior problem are due to Keel-Smith-Sogge \cite{Keel2, 
Keel3, Keel1}, in which they established almost global existence for semilinear wave equations, global existence for quasilinear wave equations under null 
condition and almost global existence for systems of quasilinear wave equations with multiple speeds in exterior domain in $\R^3$, where they assumed the 
exterior of the obstacle $\mathcal {K}$ is nontrapping or $\mathcal {K}$ is star-shaped. In these works, they introduced some kind of novel weighted spacetime 
$L^2$ estimate (referred to as KSS estimate) for inhomogeneous wave equations, which reveals great potential in studying the long time behavior for nonlinear 
wave equations, by combining the $\mathcal {O}(|x|^{-1})$ decay of solutions of wave equations. Noting that the more standard $\mathcal {O}(t^{-1})$ is much 
more difficult to obtain for the obstacle problems. The other main innovation of \cite{Keel2, Keel3, Keel1} was the adaptation of Klainerman's vector fields 
method to the exterior problems. They found that the spatial translations, spatial rotations and the scaling vector field are still well-suited for the 
exterior domain problems. Metcalfe-Sogge \cite{Metcalfe4} generalized the results in \cite{Keel3}, where they could drop the star-shaped hypothesis and handle 
non-trapping obstacles. What is more, they could improve  considerably the decay assumptions on the initial data at infinity and handle non-diagonal systems 
involving multiple wave speeds. The results in \cite{Metcalfe4} were further generalized by \cite{Metcalfe7,MNS05, Metcalfe3} to a larger class of quasilinear 
wave equations, such as the one with weaker null condition \cite{Metcalfe7}, the one with nonlinearity depending on the solution itself $u$ at the cubic level 
\cite{MNS05}.

A key ingredient of the proof in \cite{Keel3, Keel1} is the local energy decay estimate (see \cite{LMP63,Mor75}). Metcalfe-Sogge \cite{Metcalfe5} developed a 
new technique based only on energy method, relied only upon the invariance of the wave operator under translations and spatial rotations, and gave a simple 
proof for the result in \cite{Keel2}. The global existence
result for the quasilinear wave equations in the exterior of a star-shaped obstacle in $\R^n (n\ge 4)$ was also established. The main innovation is that they 
established a weighted $L^2_tL^2_x$-estimate for the perturbed wave equation analogous to the one in \cite{Keel2} for the free wave equation, which is really 
powerful for the study of quasilinear wave equations. This kind of weighted $L^2_tL^2_x$-estimate for the
perturbed wave equation (also referred to as variable-coefficient KSS estimate) was generalized to the quasilinear system in the multiple speed setting by 
Matcalfe-Sogge \cite{Metcalfe3}, and hence they proved global existence of solutions to quasilinear wave
equations satisfying the null condition in certain exterior domains in $\R^3$, by using only energy methods. We believe the ideas in \cite{Metcalfe3} can be 
applied to other related problems such as isotropic elasticity in exterior domains, see \cite{Metcalfe1, Metcalfe2, XuL25}. We also refer to some related 
global existence results in higher dimensions $(n\ge 4)$ in \cite{ShT86,Ha95,Metcalfe6} and references therein.

Another kind of initial boundary value problem for the nonlinear wave equations in exterior domains also attracts much attention, i.e., with nonlinearity 
including the solution itself. It is much more involved since we have to bound like the $L^2$ normal of the solution, which is not available from the energy 
estimate. We will not list the details here but refer the reader to \cite{Lin90, Du08, DuZ08, HeM14, HeM141, ZhZ15} and references therein.

For the initial boundary value problem of quasilinear wave equations exterior to a obstacle in $\R^2$, assuming the nonlinearity is at lest cubic, 
Katayama-Kubo-Lucente \cite{Kubo2}  proved the almost global existence for Neumann boundary condition and Kubo \cite{Kubo} established the analogous result for 
Dirichlet boundary condition, respectively. Their proofs involve some complicated weighted pointwise estimates. Very recently, for some special quadratic 
nonlinear terms $\sum_{\alpha=0}^2C^{\alpha}Q_0(u,\partial_{\alpha}u)$ with $Q_0(f,g)=\partial_tf\partial_tg-\partial_1f\partial_1g-\partial_2f\partial_2g$, 
Hou-Yin-Yuan \cite{Hou} established the global existence result for null Dirichlet boundary condition.
 For this kind of nonlinear terms, they could transform such a class of quadratically quasilinear wave equation into certain manageable one with cubic 
 nonlinear forms, by introducing two good unknowns.

\par It seems much less is known for the initial boundary value problem of quasilinear wave equation with general quadratic nonlinear term in exterior domain 
to an obstacle in $\R^2$, comparing to the corresponding exterior domain problem in higher dimensions $(n\ge 3)$ mentioned above.
Two main obstacles appear if we are trying to use the similar methods for the higher dimensional cases $(n\ge 3)$: the \textbf{first} one is that the 
variable-coefficient KSS estimate in \cite{Metcalfe5} does not hold in dimension $2$ again; the \textbf{second} one is that the local energy decay in dimension 
$2$ is much slower than that of the higher dimensional cases. Hence new techniques should be developed. In this paper, we are devoted to studying the long time 
behavior of exterior problem of quasilinear wave equations with general quadratic nonlinearity in $2$ dimensions, under the assumption of star-shaped obstacle 
and null Dirichlet boundary condition. The existence lifespan estimate is established which is "almost" sharp (with some logarithmic loss) compared to the one 
in \cite{Hor85,LiY89} for the corresponding Cauchy problem. Our proof bases on the
vector fields method developed in \cite{Klainerman1}. In particular, as in Keel-
Smith-Sogge \cite{Keel2}, we restrict to the class of admissible
vector fields. Notably the hyperbolic rotations which do not seem appropriate for
problems in exterior domains will be absent in this set,
since they have unbounded normal component
on the boundary. For the scaling vector field, the coefficients can
be large in an arbitrarily small neighborhood of the obstacle and hence we avoid to use it too. The \textbf{first} new technique introduced in this paper is to 
establish the
scaling Morawetz energy estimate
\[\|\langle t-r\rangle^{1/2}\partial
u(t,\cdot)\|_{L^2(\mathbb R^2\backslash\mathcal {K})}
\]
for the
perturbed inhomogeneous wave equations, by using the multiplier
\begin{equation*}
S+\frac12:= t\partial_t+r\partial_r+\frac12,
\end{equation*}
see section 3 below. This estimate is favorable to derive $\mathbf{t^{-\frac12}}$ decay inside the cone $\left(|x|\le \frac t2\right)$. We should remark that 
the star-shaped obstacle assumption is key for the scaling Morawetz energy estimate.
Due to the
decomposition
\[
2(t\partial_t+r\partial_r)=(t+r)(\partial_t+\partial_r)
+(t-r)(\partial_t-\partial_r),
\]
in order to control the terms caused by the scaling Morawetz multiplier
$tu_t+ru_r+u/2$, the \textbf{second} key step is that we introduce an improved scaling Morawetz
energy estimate on $\mathbb R^2$, which implies the better behavior
of good derivative $\partial_t+\partial_r$. \textbf{What is more}, by the standard weighted
Sobolev inequality and introducing two new weighted $L^2$ product estimate, we only need to use the spatial rotation once to
obtain the spatial decay $\mathcal {O}(|x|^{-1/2})$, which is key to obtain $\mathbf{t^{-\frac12}}$ decay close to the cone $(|x|\ge \frac t2)$.
And hence throughout
the proof, the spatial rotation appears at most once. We sketch our key idea and steps for the proof in the following flow chart.
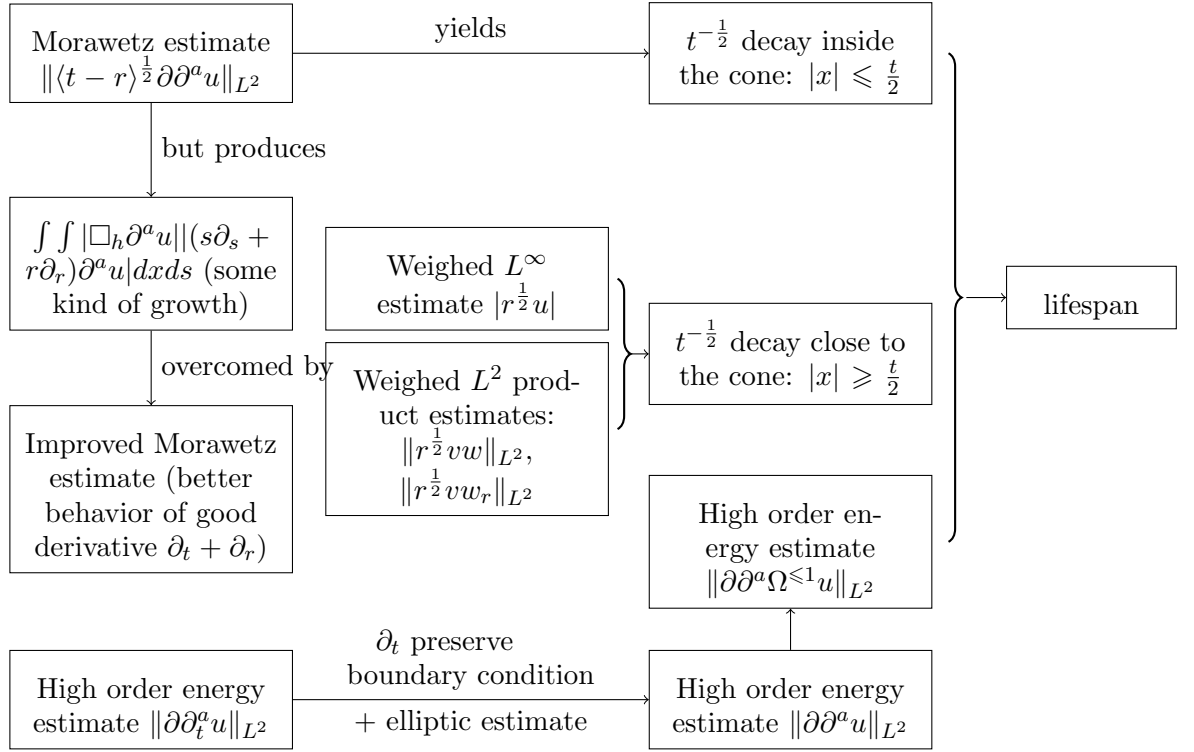
\begin{center}
\begin{tikzpicture}[    box/.style = {draw, rectangle, text width = 3.5cm, text height = 0.5cm, align = center},
    box1/.style = {draw, rectangle, text width = 2cm, text height = 0.5cm, align = center},
    box2/.style = {draw, rectangle, text width = 3cm, text height = 0.5cm, align = center},
    arrow/.style = {->, >=Stealth, thick},
    brace/.style = {decorate, decoration={brace, amplitude=5pt}, thick}
]

\node[box] (decay) at (-6,0) {Morawetz estimate $ \| \langle t-r\rangle^{\frac{1}{2}} \partial \partial^a u \|_{L^2}$};
\node[box] (middle) at (2.5,0) {$ t^{-\frac{1}{2}}$ decay inside the cone: $|x|\leq \frac{t}{2} $};

\node[box] (new) at (-6,-2.8) {$ \int\int |\Box_{h} \partial^a u| | (s\partial_s+r\partial_r)  \partial^a u| dxds $ (some kind of growth)};
\node[box] (growth) at (-6,-5.8) {Improved Morawetz estimate (better behavior of good derivative $ \partial_t + \partial_r $)};

\draw[->] (decay.east) -- (middle.west) node[midway, above] {\textcolor{black}{yields}};
\draw[->] (decay.south) -- (new.north) node[midway, right] {\textcolor{black}{but produces}};
\draw[->] (new.south) -- (growth.north) node[midway, right] {\textcolor{black}{overcomed by}};

\node[box] (decay2) at (-1.8,-3) {Weighed $ L^{\infty} $ estimate $ |r^{\frac{1}{2}} u| $ };

\node[box] (growth2) at (-1.8,-5) {Weighed $ L^{2} $ product estimates: $ \|r^{\frac{1}{2}} vw\|_{L^2} $, $ \|r^{\frac{1}{2}} vw_r\|_{L^2} $ };

\node[fit=(decay2) (growth2), inner sep=0pt, draw=none] (fitnode2) {};

\draw[brace] (0.2,-3) -- (0.2,-5) node[midway, right=15pt]{};
\node[box] (newnode) at (2.5,-4) {$t^{-\frac{1}{2}}$ decay close to the cone: $ |x|\geq \frac{t}{2} $};

\draw[->] ([xshift=-8pt]newnode.west) -- (newnode);

\node[box] (decay3) at (2.5,-6.5) {High order energy estimate $ \| \partial \partial^a \Omega^{\leq1} u  \|_{L^2} $};
\node[box] (middle3) at (2.5,-8.6) {High order energy estimate $ \| \partial \partial^a u  \|_{L^2} $};
\node[box] (growth3) at (-6,-8.6) {High order energy estimate $ \| \partial \partial^a_t u  \|_{L^2} $};

\draw[<-] (decay3.south) -- (middle3.north) node[midway, above] {\textcolor{black}{}};
\draw[<-] (middle3.west) -- (growth3.east)
    node[midway, above, align=left] {\ \ \ $\partial_t $ preserve\\boundary condition}
    node[midway, below, align=left] {  + elliptic  estimate};
\node[fit=(decay) (growth2), inner sep=0pt, draw=none] (fitnode) {};

\draw[brace] (4.6,0) -- (4.6,-6.5) node[midway, right=25pt] {};
\node[box1] (newnode11) at (6.5,-3.25) {lifespan};

\draw[->] ([xshift=-15pt]newnode11.west) -- (newnode11);
\end{tikzpicture}
\end{center}

\par
To solve system \eqref{Y1}, the initial data should satisfy
the relevant compatibility conditions: let
$J_ku=\{\nabla^au:0\leq|a|\leq k\}$. We know for a fixed $m$ and
formal $H^m$ solution of system \eqref{Y1}, we can write
$\partial_t^ku(0,\cdot)=\psi_k(J_kf,J_{k-1}g)$, $0\leq k\leq m$ for
compatibility functions $\psi_k$ depending on $J_kf$ and $J_{k-1}g$.
For $(f,g)\in H^m\times H^{m-1}$, we say that the compatibility
condition is satisfied to $m$ order, provided that $\psi_k$ vanishes
on $\partial\mathcal {K}$ for $0\leq k\leq m-1$.
\par Our main result is stated as
\begin{thm}\label{thm1}
Let $\mathcal{K}$ be a bounded and star-shaped domain in
$\mathbb R^2$, the nonlinearity $Q(\partial u,\partial^2u)$
be as above and $k\geq9$. Suppose the initial data $(f,g)\in
C_c^{\infty}(\mathbb R^2\backslash\mathcal {K})$ satisfy the
compatibility condition to infinity order. Then there exist positive
constants $A$ and $\varepsilon_0$ such that for all
$\varepsilon<\varepsilon_0$ and initial data satisfying
\begin{align}\label{Y24}
\|\nabla f\|_{H^{k+1}(\mathbb R^2\backslash\mathcal
{K})}+\|g\|_{H^{k+1}(\mathbb R^2\backslash\mathcal {K})}\leq
1,
\end{align}
system \eqref{Y1} has a unique solution $u\in
C^{\infty}([t_0,T_{\varepsilon}]\times\mathbb R^2\backslash\mathcal
{K})$, where $T_{\varepsilon}$ is the lower bound of the lifespan
and satisfies 
\begin{equation}\label{lf}
\varepsilon^2T_{\varepsilon}\ln^3T_{\varepsilon}=A,
\end{equation}
where $A>0$ is a generic constant independent of $\varepsilon$.
\end{thm}
\begin{rem}\label{rem1}
For any bounded, star-shaped domain $\mathcal {K}$ and initial data $(f,g)$ with compacted support,
by translation and scaling, without loss of generality, we may assume that $\mathcal {K}\subseteq\mathbb B_1$ is star-shaped with
respect to the origin, $supp\{f,g\}\subseteq\mathbb B_2$ and $t_0=4$.
Thus for the local solution $u(t,x)$ to system \eqref{Y1}, we have  $supp\{u(t,\cdot)\}\subseteq\{x:|x|\leq t-2\}$ for each $t\geq4$.
\end{rem}
\begin{rem}\label{rem2}
Since 
\begin{equation*}
\square
u=B_{\gamma}^{00}\partial_{\gamma}u\partial_t^2u+2B_{\gamma}^{0i}\partial_{\gamma}u\partial_t\partial_iu+B_{\gamma}^{il}\partial_{\gamma}u\partial_i\partial_lu+B(\partial
u),
\end{equation*} 
we have
\begin{align*}
\Box
u=\big(1-B_{\gamma}^{00}\partial_{\gamma}u\big)^{-1}\big(B_{\gamma}^{00}\partial_{\gamma}u\Delta
u+B(\partial
u)+2B_{\gamma}^{0i}\partial_{\gamma}u\partial_t\partial_iu+B_{\gamma}^{il}\partial_{\gamma}u\partial_i\partial_lu\big).
\end{align*}
For small data solutions to system \eqref{Y1}, by Talor's expansion,
we have $\Box u=B(\partial u)+\tilde{B}_{\gamma}^{\alpha
i}\partial_{\gamma}u\partial_{\alpha}\partial_iu+\text{cubic
terms}$. Since the cubic terms don't influence the lower bound of
the life-span of solutions to system \eqref{Y1}, we can omit these
terms. Thus without loss of generality, we assume that at least one
of the indices $\alpha,\beta$ does not equal to $0$.
\end{rem}

\begin{rem}\label{rem3}

 For the boundaryless case, there are many references on the lower bound of the lifespan estimate to the nonlinear
wave equations in two space dimensions, such as Alinhac
\cite{Alinhac1}, Hoshiga \cite{Hoshiga1,Hoshiga3}, Hoshiga and Kubo
\cite{Hoshiga2}, Katayama \cite{Katayama}, Li and Zhou
\cite{Li3,Li2} and on the upper bound of the lifespan, see
Alinhac \cite{Alinhac3,Alinhac4,Alinhac2}. With a loss of
$\ln^3T_{\varepsilon}$, our result on the lower bound of the
life-span $T_{\varepsilon}$ for solutions to system \eqref{Y1} is
sharp, as is illustrated by the finite propagation and the
counterexample of Alinhac \cite{Alinhac3,Alinhac4} in the
boundaryless case.
\end{rem}

\par The rest of this paper is organized as follows:
In section 2, we present some inequalities, which will be used in
the process of the generalized energy method. In section 3, we will establish
 some key estimates, such as scaling Morawetz energy
estimates, improved scaling Morawetz energy estimates, higher order energy
estimates and so on. In section 4, we complete the proof of Theorem
1.1 by the continuity argument.
\section{Preliminary}
\begin{lem}\label{lem1}
Denote $\langle
A\rangle=(1+|A|^2)^{1/2}$, for any smooth function $v(t,\cdot)$ defined on $\mathbb R^2$ with
sufficient decay in the infinity for each $t$, we have
\begin{align}\label{Y19}
&r^{1/2}|v(t,\cdot)|\leq
C\sum_{|a|,j\leq1}\|\nabla^a\Omega^jv(t,\cdot)\|_{L^2(\mathbb R^2)},
\\&\label{Y43}\langle t-r\rangle^{1/2}|v(t,\cdot)|\leq C\sum_{|a|\leq2}\|\langle t-r\rangle^{1/2}\nabla^av(t,\cdot)\|_{L^2(\mathbb R^2)}.
\end{align}
\end{lem}
\begin{proof}
By using a standard Sobolev inequality on $\mathbb S^1$,
$H^1(\mathbb S^1)\hookrightarrow L^{\infty}(\mathbb S^1)$, we have
\begin{align}\label{Y20}
|v(t,x)|\leq \sup_{|\theta|=1}|v(t,r\theta)|\leq
C\sum_{j\leq1}\Big(\int_{\mathbb
S^1}|\Omega^jv(t,r\theta)|^2d\theta\Big)^{1/2}.
\end{align}
A straight calculation shows that
\begin{align}\label{Y21}
|\Omega^jv(t,r\theta)|^2=-2\int_{r}^{\infty}\Omega^jv(t,\rho\theta)\partial_{\rho}\Omega^jv(t,\rho\theta)d\rho.
\end{align}
By inequalities \eqref{Y20}, \eqref{Y21} and H$\ddot{o}$lder
inequality, we have
\begin{align*}
r^{1/2}|v(t,x)|&\leq C\sum_{j\leq1}\Big(\int_{\mathbb
S^1}\int_{r}^{\infty}|\partial_{\rho}\Omega^jv(t,\rho\theta)||\Omega^jv(t,\rho\theta)|\rho
d\rho d\theta\Big)^{1/2}\\&\leq C\sum_{j\leq1}\Big(\int_{\mathbb
R^2} |\partial_r\Omega^jv(t,x)||\Omega^jv(t,x)| dx\Big)^{1/2}
\\&\leq C\sum_{j\leq1}\|\partial_r\Omega^jv\|^{1/2}_{L^2}\sum_{j\leq1}\|\Omega^jv\|^{1/2}_{L^2},
\end{align*}
which concludes the proof of inequality \eqref{Y19}. By the standard
Sobolev inequality on $\mathbb R^2$, $H^2(\mathbb
R^2)\hookrightarrow L^{\infty}(\mathbb R^2)$, obviously inequality
\eqref{Y43} holds.
\end{proof}
\begin{lem}\label{lem2}
For any functions $v(t,\cdot),w(t,\cdot)\in C_c^{\infty}(\mathbb R^2)$, we have
\begin{align*}
\|r^{1/2}vw\|_{L^2(\mathbb R^2)}\leq C\|v\|_{H^1(\mathbb
R^2)}\sum_{j\leq1}\|\Omega^jw\|_{L^2(\mathbb R^2)}.
\end{align*}
\end{lem}
\begin{proof}
Since
\begin{align}\label{a1}
r\int_{\mathbb S^1}v^2d\theta=-r\int_r^{+\infty}\int_{\mathbb S^1}2v\partial_{\rho}vd\theta d\rho\leq \int_r^{+\infty}\int_{\mathbb 
S^1}2|v||\partial_{\rho}v|\rho d\theta d\rho,
\end{align}
 a straight calculation shows that
\begin{align*}
\|r^{1/2}vw\|^2_{L^2(\mathbb R^2)}&=\int_{\mathbb
R^2}rv^2w^2dx=\int_0^{\infty}\int_{\mathbb S^1}r^2v^2w^2d\theta dr
\\&\leq C\int_0^{\infty}\int_{\mathbb S^1}r^2v^2d\theta\sum_{j\leq1}\|\Omega^jw\|^2_{L^2(\mathbb S^1)}dr
\\&\leq C\int_0^{\infty}r\sum_{j\leq1}\|\Omega^jw\|^2_{L^2(\mathbb S^1)}dr\sup_{0\leq r\leq\infty}\Big(r\int_{\mathbb S^1}v^2d\theta\Big)
\\&\leq C\sum_{j\leq1}\|\Omega^jw\|^2_{L^2(\mathbb R^2)}\int_0^{\infty}\rho\int_{\mathbb S^1}|v||\partial_{\rho}v|d\theta d\rho,
\end{align*}
which completes the proof of Lemma \ref{lem2}.
\end{proof}
\begin{lem}\label{lem3}
For any functions $v(t,\cdot),w(t,\cdot)\in C_c^{\infty}(\mathbb R^2)$ with
$supp\{v(t,\cdot),w(t,\cdot)\}\subseteq\{x:t/2\leq|x|\leq t\}$,
there is a positive constant $C$ independent on $t$ such that
\begin{align}\label{Y46}
\|r^{1/2}vw_r\|_{L^2(\mathbb R^2)}\leq C\|v\|_{H^1(\mathbb
R^2)}\big(\|\partial_rw\|_{H^2(\mathbb
R^2)}+\|\Omega w\|_{L^2(\mathbb R^2)}\big).
\end{align}
\end{lem}
\begin{proof}
Let $x_1=r\cos\theta,x_2=r\sin\theta$, where
$(r,\theta)\in[t/2,t]\times[0,2\pi]$. Thus
$\partial_{\theta}=\Omega$. By using \eqref{a1}, we have
\begin{align}\label{Y44}
\|r^{1/2}vw_r\|^2_{L^2(\mathbb R^2)}&=\int_{\mathbb
R^2}rv^2w_r^2dx=\int_{t/2}^{t}\int_0^{2\pi}r^2v^2w_r^2d\theta dr
\nonumber\\&\leq \bigg(\sup_{t/2\leq r\leq
t}\Big\{r\int_0^{2\pi}v^2d\theta \Big\}\bigg)
\int_{t/2}^tr\sup_{0\leq \theta\leq2\pi}w_r^2dr \nonumber\\&\leq
Ct\|v\|^2_{H^1(\mathbb R^2)}\int_{t/2}^t\sup_{0\leq
\theta\leq2\pi}w_r^2dr.
\end{align}
\par For any smooth function $f(\theta)$ satisfying $f(\theta)=f(\theta+2\pi)$ for any $\theta$, by Fourier series expansion, we
can expand $f(\theta)=\sum_{m\in\mathbb Z}A_me^{\textbf{i}m\theta}$
and let $
\partial_{\theta}^{2/3}f=\sum_{m\in\mathbb
Z}A_me^{\textbf{i}m\theta}(\textbf{i}m)^{2/3}$, where the Fourier
coefficients $A_m\in\mathbb C$. Thus by using the inequality
\[
\sum_{j\in \mathbb Z}|a_j||b_j|\leq \left(\sum_{j\in \mathbb
Z}|a_j|^2\right)^{1/2}(\sum_{j\in \mathbb Z}|b_j|^2\big)^{1/2},
\] 
we
have
\begin{align}\label{Y45}
|f(\theta)|\leq\sum_{m\in\mathbb Z}|A_m|&\leq\Big(\sum_{m\in\mathbb
Z}(1+m^2)^{-2/3}\Big)^{1/2}\Big(\sum_{m\in\mathbb
Z}|A_m|^2(1+m^2)^{2/3}\Big)^{1/2}\nonumber\\&\leq
C\Big(\sum_{m\in\mathbb Z}|A_m|^2(1+m^2)^{2/3}\Big)^{1/2}
\nonumber\\&\leq C\|f\|_{L_{\theta}^2([0,2\pi])}+
C\|\partial_{\theta}^{2/3}f\|_{L_{\theta}^2([0,2\pi])}.
\end{align}
\par Since the function $w(t,r,\theta)$ can be regarded as a
periodic smooth function on $\mathbb R\times\mathbb R$ satisfying
$w(t,r,\theta)=w(t,r+t/2,\theta)$ and
$w(t,r,\theta)=w(t,r,\theta+2\pi)$, by Fourier series expansion, we
have 
\[
w(t,r,\theta)=\sum_{m,n\in\mathbb
Z}A_{mn}(t)e^{\textbf{i}m\theta}e^{\textbf{i}4\pi n(r-t/2)/t},
\]
where the Fourier coefficients $A_{mn}\in\mathbb C$. Thus
\[
\partial_{\theta}^{2/3}\partial_rw(t,r,\theta)=\sum_{m,n\in\mathbb
Z}(\textbf{i}m)^{2/3}4\pi
n\textbf{i}/tA_{mn}(t)e^{\textbf{i}m\theta}e^{\textbf{i}4\pi
n(r-t/2)/t}.
\]
 When $(n/t)^2\geq m^{2/3}$, it holds
$m^{4/3}(n/t)^2\leq (n/t)^6$; while when $(n/t)^2\leq m^{2/3}$, it holds
$m^{4/3}(n/t)^2\leq m^2$. Hence by inequality \eqref{Y45}, we have
\begin{align*}
\int_{t/2}^t\sup_{0\leq \theta\leq2\pi}w_r^2dr&\leq C\|\partial_rw\|^2_{L_r^2L_{\theta}^2([t/2,t]\times[0,2\pi])}+
C\|\partial_{\theta}^{2/3}\partial_rw\|^2_{L_r^2L_{\theta}^2([t/2,t]\times[0,2\pi])}
\\&\leq C\|\partial_rw\|^2_{L_r^2L_{\theta}^2([t/2,t]\times[0,2\pi])}+Ct\sum_{m,n\in\mathbb
Z}m^{4/3}(n/t)^2|A_{mn}(t)|^2
\\&\leq C\|\partial_rw\|^2_{L_r^2L_{\theta}^2([t/2,t]\times[0,2\pi])}+Ct\sum_{m,n\in\mathbb Z}(n/t)^6|A_{mn}(t)|^2
\\&\quad\ +Ct\sum_{m,n\in\mathbb
Z}m^2|A_{mn}(t)|^2
\\&\leq
C\Big(\|\partial_rw\|^2_{L_r^2L_{\theta}^2}+\|\partial_r^3w\|^2_{L_r^2L_{\theta}^2}+\|\partial_{\theta}w\|^2_{L_r^2L_{\theta}^2}\Big),
\end{align*}
which implies inequality \eqref{Y46} by using inequality
\eqref{Y44}.
\end{proof}
\begin{lem}\label{lem4}
Let $u(t,\cdot)$ be a smooth function on $\mathbb
R^2\backslash\mathcal {K}$ with sufficient decay in the infinity for
each $t$ and $u|_{\partial\mathcal {K}}=0$. Then for any $t$ and any
multi-index $a$, we have
\begin{align}
\label{Y25}&r^{1/2}|\partial\partial^au|\leq
C\sum_{\substack{|b|+j\leq
|a|+2\\j\leq1}}\|\partial\partial^b\Omega^ju\|_{L^2(\mathbb
R^2\backslash\mathcal {K})},
\\&\label{Y30}\langle t-r\rangle^{1/2}|\partial\partial^au|\leq
C\sum_{|b|\leq |a|+2}\|\langle
t-r\rangle^{1/2}\partial\partial^bu\|_{L^2(\mathbb
R^2\backslash\mathcal {K})}.
\end{align}
\end{lem}
\begin{proof}
When $|x|\leq2$, by using the standard Sobolev Embedding Theorem:
$H^2(\Omega)\hookrightarrow L^{\infty}(\Omega)$, where
$\Omega=(\mathbb R^2\backslash\mathcal {K})\cap\mathbb B_3$,
inequality \eqref{Y25} holds. Let $\rho(x)$ be a smooth cut-off
function with $\rho(x)=1$ for $|x|\geq2$ and $\rho(x)=0$ for
$|x|\leq1$, and $v=\rho u$. Then when $|x|\geq2$, by inequality
\eqref{Y19}, we have
\begin{align*}
r^{1/2}|\partial\partial^au|=r^{1/2}|\partial\partial^av|\leq
C\sum_{\substack{|b|+j\leq
|a|+2\\j\leq1}}\|\partial\partial^b\Omega^jv\|_{L^2(\mathbb R^2)},
\end{align*}
which implies inequality \eqref{Y25} holds when $|x|\geq2$ by
Poincar\'{e} inequality. By using inequality \eqref{Y43}, we
can prove inequality \eqref{Y30} in a similar way.
\end{proof}
\section{Scaling Morawetz Energy Estimates, Improved Scaling Morawetz Energy Estimates and Energy Estimates}
In this section, we will study the following system
\begin{equation}\label{Y56}
\left \{
\begin{array}{lllll}
\Box_h u=F,\\
u|_{\partial\mathcal {K}}=0,
\end{array} \right.
\end{equation}
where
\begin{align*}
\Box_h u=\Box
u+h^{\alpha\beta}\partial_{\alpha}\partial_{\beta}u.
\end{align*}
We shall assume that $h^{\alpha\beta}$ satisfy the symmetry
conditions
\begin{align}\label{Y38}
h^{\alpha\beta}=h^{\beta\alpha},
\end{align}
as well as the size conditions
\begin{align}\label{Y39}
|h|=\sum_{\alpha,\beta=0}^2|h^{\alpha\beta}|\ll1.
\end{align}
\par Define the scaling operator by
$S=t\partial_t+r\partial_r=t\partial_t+x\cdot\nabla$. First, we
prove the following scaling Morawetz energy estimate on $\mathbb
R^2\backslash\mathcal {K}$ for the perturbed wave equation.
\begin{lem}\label{lem5}
Let $\mathcal {K}$ be bounded and star-shaped with respect to the origin.
Assume that the perturbation terms $h^{\alpha\beta}$ satisfy conditions \eqref{Y38} and \eqref{Y39}.
Then for any smooth function $u$ with $u|_{\partial\mathcal {K}}=0$ and $supp\{u(t,\cdot)\}\subseteq\{x:|x|\leq t-2\}$, we have
\begin{align}\label{Y29}
&\frac12\frac{d}{dt}\int_{\mathbb R^2\backslash\mathcal
{K}}\bigg(|t-r||\partial u|^2+r\Big(u_t+u_r+\frac
u{2r}\Big)^2+\frac{|\Omega
u|^2}r+\frac {u^2}{4r}\nonumber\\&+h^{0\beta}(2Su+u)u_{\beta}
-th^{\alpha\beta}u_{\alpha}u_{\beta}\bigg)dx
\nonumber\\&\leq\int_{\mathbb R^2\backslash\mathcal {K}}\Box_h u
\big((S+1/2)u\big)dx+C\int_{\mathbb R^2\backslash\mathcal {K}}|Sh||\partial u|^2dx
\nonumber\\&\quad+C\int_{\mathbb R^2\backslash\mathcal {K}}|\partial h||\partial u|\big(|Su|+|u|\big)dx.
\end{align}
\end{lem}
\begin{proof}
By multiplying $\Box u$ by $u_t$, we have
\begin{align}\label{Y3}
u_t\Box u&=\frac12\partial_t(u_t^2)-\nabla\cdot(u_t\nabla
u)+\nabla u\cdot\nabla u_t
\nonumber\\&=\frac12\partial_t(u_t^2+|\nabla
u|^2)-\nabla\cdot(u_t\nabla u),
\end{align}
which implies
\begin{align}\label{Y4}
(tu_t)\Box u=\frac12\partial_t(tu_t^2+t|\nabla
u|^2)-\frac12(u_t^2+|\nabla u|^2)-\nabla\cdot(tu_t\nabla u).
\end{align}
By multiplying $\Box u$ by $x\cdot\nabla u$, we have
\begin{align}\label{Y5}
(x\cdot\nabla u)\Box u&=\partial_t(u_tx\cdot\nabla
u)-\frac12x\cdot\nabla(u_t^2)-\nabla\cdot(\nabla ux\cdot\nabla u)
\nonumber\\&\quad+|\nabla u|^2+\frac12x\cdot\nabla(|\nabla u|^2)
\nonumber\\&=\partial_t(u_tx\cdot\nabla
u)+\frac12\nabla\cdot(x|\nabla
u|^2-xu_t^2)+u_t^2\nonumber\\&\quad-\nabla\cdot(\nabla ux\cdot\nabla
u).
\end{align}
By multiplying $\Box u$ by $u/2$, we have
\begin{align}\label{Y6}
(u/2)\Box
u&=\frac12\partial_t(uu_t)-\frac12u_t^2-\frac12\nabla\cdot(u\nabla
u)+\frac12|\nabla u|^2.
\end{align}
The combination of equalities \eqref{Y4}, \eqref{Y5} and \eqref{Y6}
shows that
\begin{align}\label{Y7}
&\Box
u(tu_t+ru_r+u/2)\nonumber\\&=\frac12\partial_t(tu_t^2+t|\nabla
u|^2+2ru_tu_r+uu_t)\nonumber\\&\quad+\frac12\nabla\cdot(x|\nabla
u|^2-xu_t^2)-\nabla\cdot(tu_t\nabla
u)\nonumber\\&\quad-\nabla\cdot(\nabla ux\cdot\nabla
u)-\frac12\nabla\cdot(u\nabla u).
\end{align}
\par Now let's calculate the perturbation terms. By using the symmetry conditions \eqref{Y38}, for any $\gamma=0,1,2$, we have
\begin{align*}
&\big(h^{\alpha\beta}\partial_{\alpha}\partial_{\beta}u\big)\partial_{\gamma}u
\\&=\partial_{\alpha}\big(h^{\alpha\beta}u_{\gamma}u_{\beta}\big)
-\big(\partial_{\alpha}h^{\alpha\beta}\big)u_{\gamma}u_{\beta}-h^{\alpha\beta}\partial_{\gamma}\partial_{\alpha}u\partial_{\beta}u
\\&=\partial_{\alpha}\big(h^{\alpha\beta}u_{\gamma}u_{\beta}\big)
-\big(\partial_{\alpha}h^{\alpha\beta}\big)u_{\gamma}u_{\beta}
-\frac12\partial_{\gamma}\big(h^{\alpha\beta}u_{\alpha}u_{\beta}\big)
+\frac12(\partial_{\gamma}h^{\alpha\beta})u_{\alpha}u_{\beta},
\end{align*}
which implies
\begin{align}\label{Y35}
&\big(h^{\alpha\beta}\partial_{\alpha}\partial_{\beta}u\big)\big(Su\big)
\nonumber\\&=\big(h^{\alpha\beta}\partial_{\alpha}\partial_{\beta}u\big)\big(x_{\gamma}\partial_{\gamma}u\big)
\nonumber\\&=\partial_{\alpha}\big(h^{\alpha\beta}(Su)u_{\beta}\big)
-\big(\partial_{\alpha}h^{\alpha\beta}\big)\big(Su\big)u_{\beta}
-\frac12\partial_{\gamma}\big(x_{\gamma}h^{\alpha\beta}u_{\alpha}u_{\beta}\big)
\nonumber\\&\quad+\frac12(Sh^{\alpha\beta})u_{\alpha}u_{\beta}+\frac12h^{\alpha\beta}u_{\alpha}u_{\beta}.
\end{align}
The combination of equality \eqref{Y35} and the following equality
\begin{align*}
&\big(h^{\alpha\beta}\partial_{\alpha}\partial_{\beta}u\big)\big(u/2\big)
\\&=\frac12\partial_{\alpha}\big(h^{\alpha\beta}uu_{\beta}\big)
-\frac12\big(\partial_{\alpha}h^{\alpha\beta}\big)uu_{\beta}
-\frac12h^{\alpha\beta}u_{\alpha}u_{\beta}
\end{align*}
implies that
\begin{align}\label{Y31}
&\big(h^{\alpha\beta}\partial_{\alpha}\partial_{\beta}u\big)\big((S+1/2)u\big)
\nonumber\\&=\partial_{\alpha}\big(h^{\alpha\beta}(Su+u/2)u_{\beta}\big)
-\big(\partial_{\alpha}h^{\alpha\beta}\big)\big(Su\big)u_{\beta}
-\frac12\partial_{\gamma}\big(x_{\gamma}h^{\alpha\beta}u_{\alpha}u_{\beta}\big)
\nonumber\\&\quad+\frac12(Sh^{\alpha\beta})u_{\alpha}u_{\beta}-\frac12\big(\partial_{\alpha}h^{\alpha\beta}\big)uu_{\beta}.
\end{align}
By equalities \eqref{Y7} and \eqref{Y31}, we have
\begin{align}\label{Y37}
&(\Box_h
u)(Su+u/2)\nonumber\\&=\frac12\partial_t(tu_t^2+t|\nabla
u|^2+2ru_tu_r+uu_t+h^{0\beta}(2Su+u)u_{\beta}-th^{\alpha\beta}u_{\alpha}u_{\beta})\nonumber\\&\quad+\frac12\nabla\cdot(x|\nabla
u|^2-xu_t^2)-\nabla\cdot(tu_t\nabla
u)-\nabla\cdot(\nabla ux\cdot\nabla
u)-\frac12\nabla\cdot(u\nabla u)\nonumber\\&\quad+
\partial_i\big(h^{i\beta}(Su+u/2)u_{\beta}\big)
-\frac12\nabla\cdot\big(xh^{\alpha\beta}u_{\alpha}u_{\beta}\big)
\nonumber\\&\quad-\big(\partial_{\alpha}h^{\alpha\beta}\big)\big(Su\big)u_{\beta}
+\frac12(Sh^{\alpha\beta})u_{\alpha}u_{\beta}-\frac12\big(\partial_{\alpha}h^{\alpha\beta}\big)uu_{\beta}.
\end{align}
Let $\vec{n}=(n_1,n_2)$ be the unit outward normal to $\mathcal
{K}$. The boundary condition $u|_{\partial\mathcal {K}}=0$ shows
that $u_t=0$ and $\frac{\partial u}{\partial x_i}=n_i\frac{\partial
u}{\partial n}$, $i=1,2$, on $\partial\mathcal {K}$. Since $\mathcal
{K}$ is star-shaped with respect to the origin, we know $\langle
x,\vec{n}\rangle|_{\partial\mathcal {K}}\geq0$, which implies
\begin{align*}
&-\frac12\int_{\partial\mathcal {K}}\langle x,\vec{n}\rangle|\nabla u|^2dS
+\int_{\partial\mathcal {K}}\frac{\partial u}{\partial n}\big(x\cdot\nabla u\big)dS
\\&\quad-\int_{\partial\mathcal {K}}n_ih^{il}\big(x\cdot\nabla u\big)\partial_ludS
+\frac12\int_{\partial\mathcal {K}}\langle x,\vec{n}\rangle h^{\alpha\beta}u_{\alpha}u_{\beta}dS
\\&=\int_{\partial\mathcal {K}}\Big(\frac{1-h^{il}n_in_l}2\Big)\langle x,\vec{n}\rangle\Big|\frac{\partial u}{\partial n}\Big|^2dS\\
&\geq0,
\end{align*}
where we used the size condition \eqref{Y39}.
Thus by integrating equality \eqref{Y37} with respect to $x$ on the
domain $\mathbb R^2\backslash\mathcal {K}$, we obtain
\begin{align}\label{Y28}
&\frac12\frac{d}{dt}\int_{\mathbb R^2\backslash\mathcal {K}}(tu_t^2+t|\nabla
u|^2+2ru_tu_r+uu_t\nonumber\\&\qquad\qquad\quad+h^{0\beta}(2Su+u)u_{\beta}-th^{\alpha\beta}u_{\alpha}u_{\beta})dx
\nonumber\\&\leq\int_{\mathbb R^2\backslash\mathcal {K}}(\Box_h
u)(Su+u/2)dx+C\int_{\mathbb R^2\backslash\mathcal {K}}|Sh||\partial u|^2dx
\nonumber\\&\quad+C\int_{\mathbb R^2\backslash\mathcal {K}}|\partial h||\partial u|\big(|Su|+|u|\big)dx.
\end{align}
A straight calculation shows that
\begin{align}\label{Y11}
&tu_t^2+t|\nabla u|^2+2ru_tu_r+uu_t\nonumber\\&=|t-r|(u_t^2+|\nabla
u|^2)+\frac{|\Omega u|^2}r\nonumber\\&\quad+r\Big(u_t+u_r+\frac
u{2r}\Big)^2-\frac {u^2}{4r}-uu_r,
\end{align}
and
\begin{align}\label{Y27}
\int_{\mathbb R^2\backslash\mathcal
{K}}uu_rdx=-\int_{\mathbb R^2\backslash\mathcal {K}}u^2/(2r)dx.
\end{align}
The combination of \eqref{Y28}, \eqref{Y11} and \eqref{Y27} implies
that inequality \eqref{Y29} holds.
\end{proof}
To control the right-hand side of inequality \eqref{Y29}, we introduce
the following improved scaling Morawetz energy estimate for the
perturbed wave equation in $\mathbb R^2$.
\begin{lem}\label{lem6}
Assume that the perturbation terms $h^{\alpha\beta}$ satisfy conditions \eqref{Y38} and \eqref{Y39}. Let $\sigma(t,r)=t-r$ and $A(\sigma)=(\ln|\sigma|)^{-1}$. 
For any $t\geq4$ and any smooth function $v(t,\cdot)$ defined on $\mathbb R^2$ with $supp\{v(t,\cdot)\}\subseteq\{x:|x|\leq t-2\}\cap\mathbb B_1^c$, we have
\begin{align*}
&\frac12\frac{d}{dt}\int_{\mathbb R^2}e^{A(\sigma)}\Big(|t-r||\partial v|^2
+r\Big(v_t+v_r+\frac v{2r}\Big)^2+\frac{|\Omega v|^2}r+\frac{v^2}{4r}
\nonumber\\&\qquad\qquad\qquad+\frac{v^2}{|t-r|(\ln|t-r|)^2}
+h^{0\beta}v_{\beta}(2Sv+v)-th^{\alpha\beta}v_{\alpha}v_{\beta}\Big)dx
\nonumber\\&\quad+\frac12\int_{\mathbb R^2}\frac{e^{A(\sigma)}}{|t-r|(\ln|t-r|)^2}\Big((t+r)(v_t+v_r)^2
+|t-r|\Big|\frac{\Omega}rv\Big|^2\Big)dx
\end{align*}
\begin{align}\label{Y2}
&\leq\int_{\mathbb R^2}\Box_h v
\big(e^{A(\sigma)}(S+1/2)v\big)dx
+C\int_{\mathbb R^2}|\partial h||\partial v|\big(|Sv|+|v|\big)dx
\nonumber\\&\quad+C\int_{\mathbb R^2}|Sh||\partial v|^2dx+C\int_{\mathbb R^2}(\ln|t-r|)^{-2}|h||\partial v|^2dx
\nonumber\\&\quad+C\int_{\mathbb R^2}|t-r|^{-1}(\ln|t-r|)^{-2}|h||\partial v|(|Sv|+|v|)dx
\nonumber\\&\quad+C\|r^{-1/2}|t-r|^{1/2}\partial v(t,\cdot)\|^2_{L^2(\mathbb R^2)}.
\end{align}
\end{lem}
\begin{proof}
By equality \eqref{Y37} and the decomposition $|\nabla
v|=|\partial_rv|^2+|\frac{\Omega }rv|^2$, we have
\begin{align}\label{Y8}
&\big(\Box_h
v\big)\big(e^{A(\sigma)}(S+1/2)v\big)
\nonumber\\&=\frac12\partial_t\Big(e^{A(\sigma)}\big(tv_t^2+t|\nabla
v|^2+2rv_tv_r+vv_t+h^{0\beta}v_{\beta}(2Sv+v)-th^{\alpha\beta}v_{\alpha}v_{\beta}\big)\Big)
\nonumber\\&\quad-\frac12e^{A(\sigma)}A^{'}(\sigma)\Big(tv_t^2+t|\nabla
v|^2+2rv_tv_r+vv_t+h^{0\beta}v_{\beta}(2Sv+v)-th^{\alpha\beta}v_{\alpha}v_{\beta}\Big)
\nonumber\\&\quad+\frac12\nabla\cdot\big(e^{A(\sigma)}(x|\nabla
v|^2-xv_t^2)\big)+\frac12e^{A(\sigma)}A^{'}(\sigma)(r|\nabla
v|^2-rv_t^2)
\nonumber\\&\quad-\nabla\cdot(e^{A(\sigma)}tv_t\nabla
v)-e^{A(\sigma)}A^{'}(\sigma)tv_tv_r
-\nabla\cdot(e^{A(\sigma)}\nabla
vx\cdot\nabla v)-e^{A(\sigma)}A^{'}(\sigma)rv_r^2
\nonumber\\&\quad-\frac12\nabla\cdot(e^{A(\sigma)}v\nabla
v)-\frac12e^{A(\sigma)}A^{'}(\sigma)vv_r+
\partial_i\big(e^{A(\sigma)}h^{i\beta}(Sv+v/2)v_{\beta}\big)\nonumber\\&\quad
+e^{A(\sigma)}A^{'}(\sigma)\frac{h^{i\beta}x_iv_{\beta}}r(Sv+v/2)
-\frac12\nabla\cdot\big(e^{A(\sigma)}xh^{\alpha\beta}v_{\alpha}v_{\beta}\big)
\nonumber\\&\quad-\frac12e^{A(\sigma)}A^{'}(\sigma)rh^{\alpha\beta}v_{\alpha}v_{\beta}
-e^{A(\sigma)}\big(\partial_{\alpha}h^{\alpha\beta}\big)\big(Sv\big)v_{\beta}
\nonumber\\&\quad+\frac12e^{A(\sigma)}(Sh^{\alpha\beta})v_{\alpha}v_{\beta}
-\frac12e^{A(\sigma)}\big(\partial_{\alpha}h^{\alpha\beta}\big)vv_{\beta}
\nonumber\\&=\frac12\partial_t\Big(e^{A(\sigma)}\big(tv_t^2+t|\nabla
v|^2+2rv_tv_r+vv_t+h^{0\beta}v_{\beta}(2Sv+v)-th^{\alpha\beta}v_{\alpha}v_{\beta}\big)\Big)
\nonumber\\&\quad
-\frac12e^{A(\sigma)}A^{'}(\sigma)\Big((t+r)(v_t+v_r)^2+|t-r||\frac{\Omega}rv|^2
+v(v_t+v_r)\nonumber\\&\qquad+\big(h^{0\beta}v_{\beta}
-(h^{i\beta}x_iv_{\beta})/r\big)(2Sv+v)-(t-r)h^{\alpha\beta}v_{\alpha}v_{\beta}\Big)
\nonumber\\&\quad+\frac12\nabla\cdot\big(e^{A(\sigma)}(x|\nabla
v|^2-xv_t^2)\big)-\nabla\cdot(e^{A(\sigma)}tv_t\nabla v)
\nonumber\\&\quad-\nabla\cdot(e^{A(\sigma)}\nabla vx\cdot\nabla
v)-\frac12\nabla\cdot(e^{A(\sigma)}v\nabla
v)\nonumber\\&\quad+\partial_i\big(e^{A(\sigma)}h^{i\beta}(Sv+v/2)v_{\beta}\big)
-\frac12\nabla\cdot\big(e^{A(\sigma)}xh^{\alpha\beta}v_{\alpha}v_{\beta}\big)
\nonumber\\&\quad-e^{A(\sigma)}\big(\partial_{\alpha}h^{\alpha\beta}\big)\big(Sv\big)v_{\beta}
+\frac12e^{A(\sigma)}(Sh^{\alpha\beta})v_{\alpha}v_{\beta}
\nonumber\\&\quad-\frac12e^{A(\sigma)}\big(\partial_{\alpha}h^{\alpha\beta}\big)vv_{\beta}.
\end{align}
For any $\sigma>0$, we have
$A^{'}(\sigma)=-|\sigma|^{-1}(\ln|\sigma|)^{-2}$. By integrating
equality \eqref{Y8} with respect to $x$ on the domain $\mathbb R^2$,
we obtain
\begin{align}\label{Y9}
&\frac12\frac{d}{dt}\int_{\mathbb
R^2}e^{A(\sigma)}\Big(tv_t^2+t|\nabla
v|^2+2rv_tv_r+vv_t+h^{0\beta}v_{\beta}(2Sv+v)-th^{\alpha\beta}v_{\alpha}v_{\beta}\Big)dx
\nonumber\\&+\frac12\int_{\mathbb
R^2}\frac{e^{A(\sigma)}}{|t-r|(\ln|t-r|)^2}
\Big((t+r)(v_t+v_r)^2+|t-r|\Big|\frac{\Omega}rv\Big|^2+v(v_t+v_r)\nonumber\\
&\qquad\qquad+\big(h^{0\beta}v_{\beta}
-(h^{i\beta}x_iv_{\beta})/r\big)(2Sv+v)-(t-r)h^{\alpha\beta}v_{\alpha}v_{\beta}\Big)dx
\nonumber\\&\leq\int_{\mathbb R^2}\big(\Box_h v\big)
\big(e^{A(\sigma)}(S+1/2)v\big)dx+C\int_{\mathbb R^2}|Sh||\partial
v|^2dx \nonumber\\&\quad+C\int_{\mathbb R^2}|\partial h||\partial
v|\big(|Sv|+|v|\big)dx.
\end{align}
Since
\begin{align*}
&\int_{\mathbb R^2}\frac{e^{A(\sigma)}}{|t-r|(\ln|t-r|)^2}v(v_t+v_r)dx
\\&=\frac12\frac{d}{dt}\int_{\mathbb R^2}\frac{e^{A(\sigma)}}{|t-r|(\ln|t-r|)^2}v^2dx
-\frac12\int_{\mathbb R^2}\frac{e^{A(\sigma)}}{|t-r|(\ln|t-r|)^2}\frac{v^2}rdx,
\end{align*}
by inequality \eqref{Y9} and using similar argument to equalities
\eqref{Y11}, \eqref{Y27}, to complete the proof of this lemma, it
suffices to estimate 
\[
\int_{\mathbb
R^2}\frac{e^{A(\sigma)}}{|t-r|(\ln|t-r|)^2}\frac{v^2}rdx.
\]
 A
straight calculation shows that
\begin{align*}
\int_0^{\infty}\frac{v^2}{|t-r|(\ln|t-r|)^2}dr&=\int_0^{t-2}v^2d\frac1{\ln|t-r|}
\leq \int_0^{\infty}\frac{2|v||v_r|}{\ln|t-r|}dr
\\&\leq 2\Big(\int_0^{\infty}\frac{v^2}{|t-r|(\ln|t-r|)^2}dr\Big)^{1/2}\Big(\int_0^{\infty}|t-r|v_r^2dr\Big)^{1/2},
\end{align*}
which implies
\begin{align*}
\int_{\mathbb R^2}\frac{e^{A(\sigma)}v^2}{r|t-r|(\ln|t-r|)^2}dx
\leq C\int_{\mathbb R^2}r^{-1}|t-r|v_r^2dx.
\end{align*}
This completes the proof of Lemma \ref{lem6}.
\end{proof}
To assure the positivity of the scaling Morawetz energy and the
improved scaling Morawetz energy for the perturbed wave equation, we
assume the perturbation terms $h^{\alpha\beta}$ satisfy the
following two inequalities 
\begin{align}\label{Y41}
&\frac12\Big(|t-r||\partial u|^2+r\Big(u_t+u_r+\frac
u{2r}\Big)^2+\frac{|\Omega
u|^2}r+\frac {u^2}{4r}\Big)\nonumber\\&\leq |t-r||\partial u|^2+r\Big(u_t+u_r+\frac
u{2r}\Big)^2+\frac{|\Omega
u|^2}r+\frac {u^2}{4r}\nonumber\\&\quad\ +h^{0\beta}(2Su+u)u_{\beta}-th^{\alpha\beta}u_{\alpha}u_{\beta}
\nonumber\\&\leq2\Big(|t-r||\partial u|^2+r\Big(u_t+u_r+\frac
u{2r}\Big)^2+\frac{|\Omega
u|^2}r+\frac {u^2}{4r}\Big),
\end{align}
and
\begin{align}\label{Y47}
&\frac12\Big(|t-r||\partial v|^2+r\Big(v_t+v_r+\frac
v{2r}\Big)^2+\frac{|\Omega v|^2}r+\frac
{v^2}{4r}+\frac{v^2}{|t-r|(\ln|t-r|)^2}\Big) \nonumber\\&\leq
e^{A(\sigma)}\Big(|t-r||\partial v|^2+r\Big(v_t+v_r+\frac
v{2r}\Big)^2+\frac{|\Omega v|^2}r+\frac
{v^2}{4r}\nonumber\\&\qquad\qquad
+\frac{v^2}{|t-r|(\ln|t-r|)^2}+h^{0\beta}(2Sv+v)v_{\beta}-th^{\alpha\beta}v_{\alpha}v_{\beta}\Big)
\nonumber\\&\leq \widetilde{C}\Big(|t-r||\partial
v|^2+r\Big(v_t+v_r+\frac v{2r}\Big)^2+\frac{|\Omega v|^2}r+\frac
{v^2}{4r}+\frac{v^2}{|t-r|(\ln|t-r|)^2}\Big),
\end{align}
for any $|x|\leq t-2$ and any functions $u$,$v$, where
$\widetilde{C}=2\exp{(ln2)^{-1}}$.
\par Since $\partial_t$ preserves the null boundary condition, we may obtain the
higher-order scaling Morawetz energy estimates involving the
temporal derivative $\partial_t$.
\begin{lem}\label{lem7}
Let $t\geq4$ and $k\in\mathbb{N}^{+}$. Assume the perturbation terms
$h^{\alpha\beta}$ satisfy the conditions \eqref{Y38}, \eqref{Y39}
and \eqref{Y41}, \eqref{Y47}. For any local
smooth solution $u(t,\cdot)$ to system \eqref{Y56} with the initial
data satisfying condition \eqref{Y24}, where
$supp\{u(t,\cdot)\}\subseteq\{x:x\leq t-2\}$, we then have
\begin{align}\label{Y13}
&\sum_{j\leq k-1}\||t-r|^{1/2}\partial\partial_t^j
u(t,\cdot)\|^2_{L^2(\mathbb R^2\backslash\mathcal
{K})}\nonumber\\&\leq C\varepsilon^2+C\sum_{j\leq
k-1}\int_4^t\int_{\mathbb R^2\backslash\mathcal
{K}}|\Box_h\partial_s^ju(s,x)|\big(|S\partial_s^ju(s,x)|+|\partial_s^ju(s,x)|\big)dxds
\nonumber\\&\quad+C\sum_{j\leq k-1}\int_4^t\int_{\mathbb
R^2\backslash\mathcal {K}}|\partial
h||\partial\partial_s^ju(s,x)|(|S\partial_s^ju(s,x)|+|\partial_s^ju(s,x)|)dxds
\nonumber\\&\quad+C\sum_{j\leq k-1}\int_4^t\int_{\mathbb
R^2\backslash\mathcal {K}}|Sh||\partial\partial_s^ju(s,x)|dxds.
\end{align}
\end{lem}
\begin{proof}
Replacing $u$ by $\partial_t^ju$ in inequality \eqref{Y29} and
integrating the resulted inequality with respect to $t$ on $[4,t]$, then
 Lemma \ref{lem7} holds.
\end{proof}
Next let us come to the higher-order scaling Morawetz energy estimates
involving the temporal and spatial derivatives. Precisely we obtain
the following lemma.
\begin{lem}\label{lem8}
Suppose $k\geq2$ and $t\geq4$. For any smooth function $u(t,\cdot)$
defined on $\mathbb R^2\backslash\mathcal {K}$ with compacted
support and $u|_{\partial\mathcal {K}}=0$, we have
\begin{align}\label{Y23}
&\sum_{|a|\leq k-1}\|\langle
t-r\rangle^{1/2}\partial\partial^au\|_{L^2(\mathbb
R^2\backslash\mathcal {K})}\nonumber\\&\leq C\sum_{|a|\leq
k-2}\|\langle t-r\rangle^{1/2}\partial^a\Box u\|_{L^2(\mathbb
R^2\backslash\mathcal {K})}\nonumber\\&\quad+C\sum_{j\leq
k-1}\|\langle t-r\rangle^{1/2}\partial
\partial_t^ju\|_{L^2(\mathbb R^2\backslash\mathcal {K})}.
\end{align}
\end{lem}
\begin{proof}
\par Let $\alpha(x)$ be a smooth cut-off function with $\alpha(x)=1$ for
$|x|\leq2$ and $\alpha(x)=0$ for $|x|\geq5/2$. Let $\tilde{u}=\alpha
u$. Since $\Delta \tilde{u}=\alpha(\Box
u-\partial_t^2u)+2\nabla\alpha\cdot\nabla u+u\Delta\alpha$ and
$\partial_t$ preserves the null boundary condition, by the standard
elliptic regularity estimate and Poincar\'{e} inequality, for any
$k\geq2$, $|b|+j\leq k$ and $|b|\geq2$, we obtain
\begin{align}\label{Y18}
&\|\nabla^b
\partial_t^ju\|_{L^2(|x|\leq 2)}
\nonumber\\\leq&\|\nabla^b
\partial_t^j\tilde{u}\|_{L^2(|x|\leq 5/2)}\nonumber\\
\leq&
C\|\Delta\partial_t^j \tilde{u}\|_{H^{|b|-2}(|x|\leq
5/2)}+C\sum_{j\leq k-1}\|\partial
\partial_t^ju\|_{L^2(|x|\leq 5/2)}\nonumber\\
\leq &C\sum_{|a|\leq
k-2}\|\partial^a\Box u\|_{L^2(|x|\leq 5/2)}+C\sum_{j\leq
k-1}\|\partial\partial_t^ju\|_{L^2(|x|\leq 5/2)} \nonumber\\
&+C\sum_{\substack{|c|+j\leq k\\|c|\leq
|b|-1}}\|\nabla^c\partial_t^ju\|_{L^2(|x|\leq 5/2)}.
\end{align}
The last term on the right-hand side of inequality \eqref{Y18} can
be dealt with by using some new cut-off functions and an induction
argument. Finally we obtain
\begin{align*}
&\sum_{\substack{|b|+j\leq k\\|b|\geq2}}\|\nabla^b
\partial_t^ju\|_{L^2(|x|\leq 2)}\\&\leq C\sum_{|a|\leq
k-2}\|\partial^a\Box u\|_{L^2(|x|\leq 3)}+C\sum_{j\leq
k-1}\|\partial\partial_t^ju\|_{L^2(|x|\leq 3)},
\end{align*}
which implies
\begin{align}\label{Y15}
&\sum_{|a|\leq k-1}\|\partial\partial^au\|_{L^2(|x|\leq
2)}\nonumber\\&\leq C\sum_{|a|\leq k-2}\|\partial^a\Box
u\|_{L^2(|x|\leq 3)}+C\sum_{j\leq
k-1}\|\partial\partial_t^ju\|_{L^2(|x|\leq 3)}.
\end{align}
Next we estimate $\sum_{|a|\leq k-1}\|\langle
t-r\rangle^{1/2}\partial\partial^au\|_{L^2(|x|\geq2)}$. Let
$\psi(x)$ be a smooth cut-off function with $\psi(x)=1$ for
$|x|\geq2$ and $\psi(x)=0$ for $|x|\leq1$. Let $v=\psi u$. For any
smooth function $f$ defined on $\mathbb R^2$ with sufficient decay
at the infinity, by integration by parts and Young's inequality, we
have
\begin{align}\label{Y16}
&\|\langle t-r\rangle^{1/2}\nabla^2f\|_{L^2(\mathbb R^2)}
\nonumber\\=&\sum_{i,j=1}^2\Big(\int\langle
t-r\rangle\big(\partial_i\partial_jf\big)^2dx\Big)^{1/2}
\nonumber\\\leq& \|\langle t-r\rangle^{1/2}\Delta f\|_{L^2(\mathbb
R^2)}+C\|\nabla f\|_{L^2(\mathbb R^2)}.
\end{align}
By inequalities \eqref{Y15}, \eqref{Y16} and Poincar\'{e} inequality, we
have
\begin{align}\label{Y12} &\sum_{|a|\leq k-1}\|\langle
t-r\rangle^{1/2}\partial\partial^au\|_{L^2(|x|\geq
2)}\leq\sum_{|a|\leq k-1}\|\langle
t-r\rangle^{1/2}\partial\partial^av\|_{L^2(\mathbb
R^2)}\nonumber\\&\leq C\sum_{\substack{|b|+j\leq
k\\|b|\geq2}}\|\langle t-r\rangle^{1/2}\nabla^b
\partial_t^jv\|_{L^2(\mathbb R^2)}+C\sum_{j\leq k-1}\|\langle t-r\rangle^{1/2}\partial
\partial_t^jv\|_{L^2(\mathbb R^2)}
\nonumber\\&\leq C\sum_{\substack{|b|+j\leq k,|b|\geq2\\|c|\leq
|b|-2}}\|\langle t-r\rangle^{1/2}\nabla^c\partial_t^j\Delta
v\|_{L^2(\mathbb R^2)}+C\sum_{j\leq k-1}\|\langle
t-r\rangle^{1/2}\partial
\partial_t^ju\|_{L^2(\mathbb R^2\backslash\mathcal {K})}\nonumber\\&\quad
+C\sum_{\substack{|b|+j\leq k,|b|\geq2\\|c|\leq |b|-1}}\|\langle
t-r\rangle^{1/2}\nabla^c
\partial_t^jv\|_{L^2(\mathbb R^2)}
\nonumber\\&\leq C\sum_{|a|\leq k-2}\|\langle
t-r\rangle^{1/2}\partial^a\Box u\|_{L^2(\mathbb
R^2\backslash\mathcal {K})}+C\sum_{j\leq k-1}\|\langle
t-r\rangle^{1/2}\partial\partial_t^ju\|_{L^2(\mathbb
R^2\backslash\mathcal {K})},
\end{align}
where we used $\Delta v=\psi(\Box
u-\partial_t^2u)+2\nabla\psi\cdot\nabla u+u\Delta\psi$ and an
induction argument on $|b|$. The combination of estimates
\eqref{Y15} and \eqref{Y12} shows that Lemma \ref{lem8} holds.
\end{proof}
We also need the higher-order energy estimates involving both
$\partial$ and $\Omega$, where the spatial rotation will appear at
most once. To this end, we first derive the standard higher-order energy estimates involving only the temporal derivative $\partial_t$, which can be obtained 
directly since $\partial_t$ preserves the boundary condition.
\begin{lem}\label{lem10}
Let $t\geq4$ and $k\in\mathbb N$.  Assume the perturbation terms
$h^{\alpha\beta}$ satisfy the conditions \eqref{Y38} and \eqref{Y39}.
Then for any local smooth solution $u(t,\cdot)$ to system \eqref{Y56}
with the initial data satisfying condition \eqref{Y24}, where
$supp\{u(t,\cdot)\}\subseteq\{x:x\leq t-2\}$, we have
\begin{align*}
\sum_{j\leq
k+1}\|\partial\partial_t^ju(t,\cdot)\|^2_{L^2(\mathbb
R^2\backslash\mathcal {K})}&\leq
C\varepsilon^2+C\sum_{j\leq k+1}\int_4^t\int_{\mathbb R^2\backslash\mathcal
{K}}|\Box_h
\partial_s^ju||\partial_s\partial_s^ju|dxds
\nonumber\\&\quad\ +C\sum_{j\leq
k+1}\int_4^t\int_{\mathbb R^2\backslash\mathcal
{K}}|\partial h||\partial\partial_s^ju|^2dxds.
\end{align*}
\end{lem}
To establish the standard higher-order energy estimates involving both the
temporal and spatial derivatives, we need the following
elliptic estimates.
\begin{lem}[Elliptic Regularity Estimates]\label{lem11}
Let $\Omega=\mathbb B_{t}\cap(\mathbb R^2\setminus\mathcal {K})$,
where $t$ is a positive constant and $t\geq4$. Consider the
following system
\begin{equation}\label{Y52}
\left \{
\begin{array}{llll}
-\Delta w=f \quad in \quad \Omega,  \\
w=0 \quad on \quad \partial\Omega.
\end{array} \right.
\end{equation}
For any solution $w$ to system \eqref{Y52}, there exists a positive
constant $C_0$ such that
\begin{align}\label{Y54}
\|\nabla^2w\|_{L^2(\Omega)}\leq C_0\big(\|f\|_{L^2(\Omega)}+\|\nabla
w\|_{L^2(\Omega)}\big),
\end{align}
where $C_0$ is a generic constant independent of $t$. Moreover, for any $k\geq2$, there
exists a positive constant $C_k$ independent of $t$ such that
\begin{align}\label{Y55}
\sum_{|a|=2}^k\|\nabla^aw\|_{L^2(\Omega)}\leq
C_k\big(\|f\|_{H^{k-2}(\Omega)}+\|\nabla w\|_{L^2(\Omega)}\big).
\end{align}
\end{lem}
\begin{proof}
When $t\leq6$, by the standard elliptic regularity estimates, we
have
\begin{align*}
\|w\|_{H^2(\Omega)}\leq
C\big(\|f\|_{L^2(\Omega)}+\|w\|_{L^2(\Omega)}\big).
\end{align*}
By Poincar\'{e} inequality, we have $\|w\|_{L^2(\Omega)}\leq C\|\nabla
w\|_{L^2(\Omega)}$, which implies inequality \eqref{Y54} holds for
any $t\leq6$. When $t\geq6$, let $\widetilde{w}(x)=36/t^2w(tx/6)$
and $\widetilde{\Omega}=\{x:x=6y/t,y\in\Omega\}$. Thus $-\Delta
\widetilde{w}(x)=f(tx/6)$, $\nabla\widetilde{w}(x)=6/t(\nabla
w)(tx/6)$ and $\nabla^2\widetilde{w}(x)=(\nabla^2 w)(tx/6)$. By the
standard elliptic regularity estimates and Poinare inequality,
\begin{align*}
\|\nabla^2\widetilde{w}\|_{L^2(\widetilde{\Omega})}\leq
C\big(\|f(tx/6)\|_{L^2(\widetilde{\Omega})}+\|\nabla
\widetilde{w}\|_{L^2(\widetilde{\Omega})}\big),
\end{align*}
which implies inequality \eqref{Y54} holds by making a change of
variable $tx/6\rightarrow x$.
\par Similarly, for $|a|\geq2$, we obtain
\begin{align*}
\|\nabla^aw\|_{L^2(\Omega)}\leq
C\big(\|f\|_{H^{|a|-2}(\Omega)}+\|\nabla w\|_{L^2(\Omega)}\big),
\end{align*}
which implies inequality \eqref{Y55} holds.
\end{proof}
The combination of Lemma \ref{lem10} and \ref{lem11} implies that the following lemma holds.
\begin{lem}\label{lem12}
Let $t\geq4$ and $k\in\mathbb N$.  Assume the perturbation terms
$h^{\alpha\beta}$ satisfy the conditions \eqref{Y38} and \eqref{Y39}.
Then for any local smooth solution $u(t,\cdot)$ to system \eqref{Y56}
with the initial data satisfying condition \eqref{Y24}, where
$supp\{u(t,\cdot)\}\subseteq\{x:x\leq t-2\}$, we have
\begin{align*}
\sum_{|a|\leq k+1}\|\partial\partial^au(t,\cdot)\|^2_{L^2(\mathbb
R^2\backslash\mathcal {K})}&\leq
C\varepsilon^2+C\sum_{j\leq k+1}\int_4^t\int_{\mathbb R^2\backslash\mathcal
{K}}|\Box_h
\partial_s^ju||\partial_s\partial_s^ju|dxds
\nonumber\\&\quad\ +C\sum_{j\leq k+1}\int_4^t\int_{\mathbb
R^2\backslash\mathcal {K}}|\partial h||\partial\partial_s^ju|^2dxds
\nonumber\\&\quad\ +C\sum_{|a|\leq k}\|\partial^a\Box
u(t,\cdot)\|^2_{L^2(\mathbb R^2\backslash\mathcal {K})}.
\end{align*}
\end{lem}
\begin{proof}
A straight calculation shows that
\begin{align*}
&\quad\sum_{|a|\leq k+1}\|\partial\partial^au(t,\cdot)\|^2_{L^2(\mathbb
R^2\backslash\mathcal {K})}
\\&\leq  \sum_{\substack{|b|+j\leq k+2\\|b|\geq2}}\|\nabla^b\partial_t^ju(t,\cdot)\|^2_{L^2(\mathbb
R^2\backslash\mathcal {K})}+\sum_{j\leq k+1}\|\partial\partial_t^ju(t,\cdot)\|^2_{L^2(\mathbb
R^2\backslash\mathcal {K})}.
\end{align*}
By Lemma \ref{lem11} and using $\Delta u=\partial_t^2u-\Box u$,
\begin{align*}
&\quad\sum_{\substack{|b|+j\leq k+2\\|b|\geq2}}\|\nabla^b\partial_t^ju(t,\cdot)\|^2_{L^2(\mathbb
R^2\backslash\mathcal {K})}
\\&\leq C\sum_{0\leq j\leq k}\|\Delta\partial_t^ju(t,\cdot)\|^2_{H^{k-j}(\mathbb
R^2\backslash\mathcal {K})}+C\sum_{0\leq j\leq k}\|\nabla\partial_t^ju(t,\cdot)\|^2_{L^2(\mathbb
R^2\backslash\mathcal {K})}
\\&\leq C\sum_{|a|\leq k}\|\partial^a\Box
u(t,\cdot)\|^2_{L^2(\mathbb R^2\backslash\mathcal {K})}+C\sum_{|b|+j\leq k}\|\nabla^b\partial_t^{j}\partial_t^2u(t,\cdot)\|^2_{L^2(\mathbb
R^2\backslash\mathcal {K})}
\\&\quad+C\sum_{0\leq j\leq k+1}\|\partial\partial_t^ju(t,\cdot)\|^2_{L^2(\mathbb
R^2\backslash\mathcal {K})},
\end{align*}
where the second term on the right-side of the above inequality becomes the third term by induction with respect to $|b|+j$.
\end{proof}

With Lemma \ref{lem10} and \ref{lem12} in hand, we finally come to the higher-order energy estimate involving both $\partial$ and $\Omega$, thus
\begin{lem}\label{lem9}
Let $t\geq4$ and $k\geq2$.  Assume the perturbation terms
$h^{\alpha\beta}$ satisfy the conditions \eqref{Y38}, \eqref{Y39}.
For any local smooth solution $u(t,\cdot)$ to system \eqref{Y56}
with the initial data satisfying condition \eqref{Y24}, where
$supp\{u(t,\cdot)\}\subseteq\{x:x\leq t-2\}$, we have
\begin{align*}
\sum_{\substack{|a|+j\leq
k-1\\j\leq1}}\|\partial\partial^a\Omega^ju(t,\cdot)\|^2_{L^2(\mathbb
R^2\backslash\mathcal {K})}&\leq
C\varepsilon^2+C\sum_{\substack{|a|+j\leq k-1\\j\leq1}}\|\Box_h
\partial^a\Omega^ju\|^2_{L^1_sL^2_x([4,t]\times\mathbb R^2\backslash\mathcal {K})}
\nonumber\\&\quad\ +C\sum_{\substack{|a|+j\leq
k-1\\j\leq1}}\int_4^t\int_{\mathbb R^2\backslash\mathcal
{K}}|\partial h||\partial\partial^a\Omega^ju|^2dxds
\nonumber\\&\quad\ +C\sum_{|a|\leq
k-1}\int_4^t\|\partial\partial^au(s,\cdot)\|^2_{L^2(|x|\leq2)}ds
\nonumber\\&\quad\ +C\sum_{|a|\leq
k-1}\|\partial\partial^au(t,\cdot)\|^2_{L^2(|x|\leq2)}.
\end{align*}
\end{lem}
\begin{proof}
Let $\chi(x)$ be a smooth cut-off function with $\chi(x)=1$ for
$|x|\leq1$ and $\chi(x)=0$ for $|x|\geq2$. Let $u_1=\chi u$ and
$u_2=(1-\chi)u$. Hence for any $t\geq4$, we have
\begin{align}\label{Y22}
&\sum_{\substack{|a|+j\leq
k-1\\j\leq1}}\|\partial\partial^a\Omega^ju\|_{L^2(\mathbb
R^2\backslash\mathcal {K})}\nonumber\\&\leq C\sum_{|a|\leq
k-1}\|\partial\partial^au\|_{L^2(|x|\leq
2)}+\sum_{\substack{|a|+j\leq
k-1\\j\leq1}}\|\partial\partial^a\Omega^ju_2\|_{L^2(\mathbb R^2)}.
\end{align}
By taking the $L^2(\mathbb R^2)$ inner product of
$\Box_h\partial^a\Omega^j u_2$ with
$\partial_t\partial^a\Omega^ju_2$, we have
\begin{align}\label{Y17}
\sum_{\substack{|a|+j\leq
k-1\\j\leq1}}\|\partial\partial^a\Omega^ju_2\|^2_{L^2(\mathbb R^2)}
&\leq C\varepsilon^2+C\sum_{\substack{|a|+j\leq
k-1\\j\leq1}}\int_4^t\int_{\mathbb R^2}|\Box_h\partial^a\Omega^j
u_2||\partial_s\partial^a\Omega^ju_2|dxds \nonumber\\&\quad\
+C\sum_{\substack{|a|+j\leq k-1\\j\leq1}}\int_4^t\int_{\mathbb
R^2}|\partial h||\partial\partial^a\Omega^ju_2|^2dxds
\nonumber\\&\leq C\varepsilon^2+C\sum_{\substack{|a|+j\leq
k-1\\j\leq1}}\int_4^t\int_{\mathbb R^2}|\Box_h\partial^a\Omega^j
u||\partial_s\partial^a\Omega^ju_2|dxds\nonumber\\&\quad\
+C\sum_{\substack{|a|+j\leq k-1\\j\leq1}}\int_4^t\int_{\mathbb
R^2\backslash\mathcal {K}}|\partial
h||\partial\partial^a\Omega^ju|^2dxds \nonumber\\&\quad\
+C\sum_{|a|\leq
k-1}\int_4^t\|\partial\partial^au(s,\cdot)\|^2_{L^2(|x|\leq2)}ds,
\end{align}
where we used the symmetry condition \eqref{Y38} and the size
condition \eqref{Y39}. The combination of inequalities \eqref{Y22}
and \eqref{Y17} implies that Lemma \ref{lem9} holds.
\end{proof}

\section{The Proof of Theorem \ref{thm1}}
\par A functional $M(t)$ will be introduced in the following theorem below, which collects all the estimates established in section 3.
\begin{thm}\label{thm2}
Let $t\geq4$ and $k\geq9$. Suppose that
$\sup_{t,x}(t+|x|)(|h|+|\partial h|)\leq1$. Assume the perturbation
terms $h^{\alpha\beta}$ satisfy the conditions \eqref{Y38},
\eqref{Y39}, \eqref{Y41} and \eqref{Y47}.
Then for any local smooth solution $u(t,\cdot)$ to system \eqref{Y56}
with the initial data satisfying condition \eqref{Y24}, where
$supp\{u(t,\cdot)\}\subseteq\{x:x\leq t-2\}$, there exists a
positive constant $C_0$ such that
\begin{align}\label{Y14}
M(t)&\leq C_0\varepsilon^2\ln t+I+II+III+IV+V,
\end{align}
where
\begin{align*}
&M(t)=\sum_{|a|\leq k-1}\int_4^t\int_{|x|\geq
2s/3}\frac{(s+r)\big((\partial_s+\partial_r)\partial^au(s,x)\big)^2}{|s-r|(\ln|s-r|)^2}dxds
\nonumber\\&\qquad\qquad+\sum_{|a|\leq
k-1}\sup_{4\leq s\leq t}\||s-r|^{-1/2}(\ln|s-r|)^{-1}\partial^au(s,\cdot)\|^2_{L^2(|x|\geq
2s/3)}\nonumber\\&\qquad\qquad+\sum_{|a|\leq k-1}\sup_{4\leq s\leq t}\|r^{1/2}(\partial_t+\partial_r)\partial^av(s,\cdot)\|^2_{L^2(|x|\geq 2s/3)}
\nonumber\\&\qquad\qquad+\sum_{|a|\leq
k-1}\sup_{4\leq s\leq t}\||s-r|^{1/2}\partial\partial^au(s,\cdot)\|^2_{L^2(\mathbb
R^2\backslash\mathcal {K})}+\sup_{4\leq s\leq t}\sum_{j\leq1}\Big\|\frac{\Omega^ju(s,\cdot)}{r^{1/2}}\Big\|^2_{L^2(\mathbb
R^2\backslash\mathcal {K})}
\nonumber\\&\qquad\qquad+\sum_{\substack{|a|+j\leq
k-1\\j\leq1}}\sup_{4\leq s\leq t}\|\partial\partial^a\Omega^ju(s,\cdot)\|^2_{L^2(\mathbb
R^2\backslash\mathcal {K})}+\sum_{|a|\leq
k+1}\sup_{4\leq s\leq t}\|\partial\partial^au(s,\cdot)\|^2_{L^2(\mathbb
R^2\backslash\mathcal {K})},
\\&I=C\sum_{\substack{|a|+j\leq k-1\\j\leq1}}\|\Box_h
\partial^a\Omega^ju\|^2_{L^1_sL^2_x([4,t]\times\mathbb R^2\backslash\mathcal {K})}
 +C\sum_{\substack{|a|+j\leq
k-1\\j\leq1}}\int_4^t\int_{\mathbb R^2\backslash\mathcal
{K}}|\partial h||\partial\partial^a\Omega^ju|^2dxds,
\end{align*}
\begin{align*}
&II=C\sum_{j\leq k+1}\int_4^t\int_{\mathbb R^2\backslash\mathcal
{K}}|\Box_h
\partial_s^ju||\partial_s\partial_s^ju|dxds
\nonumber\\&\qquad+C\sum_{j\leq k+1}\int_4^t\int_{\mathbb
R^2\backslash\mathcal {K}}|\partial h||\partial\partial_s^ju|^2dxds
\nonumber\\&\qquad+C\sum_{|a|\leq k}\|\partial^a\Box
u(t,\cdot)\|^2_{L^2(\mathbb R^2\backslash\mathcal {K})},
\\&III=C\sum_{|a|\leq
k-1}\|r^{1/2}|s-r|^{1/2}(\ln|s-r|)\Box_h\partial^au\|^2_{L^2_sL^2_x([4,t]\times\{|x|\geq
s/2\})} \nonumber\\&\qquad\ +C\sum_{|a|\leq
k-1}\||s-r|^{1/2}(\ln|s-r|)\Box_h\partial^au\|^2_{L_s^1L_x^2([4,t]\times\{|x|\geq
s/2\})}\nonumber\\&\qquad\ +C\sum_{|a|\leq k-2}\|\langle
t-r\rangle^{1/2}\partial^a\Box u(t,\cdot)\|^2_{L_x^2(\mathbb
R^2\backslash\mathcal {K})},
\\&IV=C\sum_{|a|\leq k-1}\int_4^t\int_{\mathbb R^2\backslash\mathcal {K}}(\ln|s-r|)^{-2}|h||\partial\partial^au|^2dxds
\nonumber\\&\qquad\ +C\sum_{|a|\leq k-1}\int_4^t\int_{\mathbb
R^2\backslash\mathcal
{K}}|s-r|^{-1}(\ln|s-r|)^{-2}|h||\partial\partial^au|(|S\partial^au|+|\partial^au|)dxds,
\\&V=C\sum_{|a|\leq k-2}\int_4^ts^{-1}\|\langle
s-r\rangle^{1/2}\partial^a\Box u(s,\cdot)\|^2_{L_x^2(\mathbb
R^2\backslash\mathcal {K})}ds
\nonumber\\&\qquad\ +C\sum_{|a|\leq k-1}\int_4^t\int_{\mathbb R^2\backslash\mathcal
{K}}|\Box_h\partial^au(s,x)|\big(|S\partial^au(s,x)|+|\partial^au(s,x)|\big)dxds
\nonumber\\&\qquad\ +C\sum_{j\leq
k-1}\int_4^ts^{-1}\int_4^s\int_{\mathbb R^2\backslash\mathcal
{K}}|\Box_h\partial_{\tau}^ju(\tau,x)|\big(|S\partial_{\tau}^ju(\tau,x)|+|\partial_{\tau}^ju(\tau,x)|\big)dxd\tau
ds
\nonumber\\&\qquad\ +C\sum_{|a|\leq k-1}\int_4^t\int_{\mathbb
R^2\backslash\mathcal {K}}|\partial
h||\partial\partial^au|\big(|S\partial^au|+|\partial^au|\big)dxds
\nonumber\\&\qquad\
+C\sum_{j\leq k-1}\int_4^ts^{-1}\int_4^s\int_{\mathbb
R^2\backslash\mathcal {K}}|\partial
h||\partial\partial_{\tau}^ju|\big(|S\partial_{\tau}^ju|+|\partial_{\tau}^ju|\big)dxd\tau
ds
\nonumber\\&\qquad\ +C\sum_{|a|\leq k-1}\int_4^t\int_{\mathbb
R^2\backslash\mathcal {K}}|Sh||\partial\partial^au(s,x)|^2dxds
\nonumber\\&\qquad\ +C\sum_{j\leq
k-1}\int_4^ts^{-1}\int_4^s\int_{\mathbb R^2\backslash\mathcal
{K}}|Sh||\partial\partial_{\tau}^ju|^2dxd\tau ds.
\end{align*}
\end{thm}
\begin{proof}
By Lemma \ref{lem7}, Lemma \ref{lem8}, Lemma \ref{lem12} and Lemma \ref{lem9}, it is easy to see that the last
four terms of $M(t)$ is controlled by the right-hand side of
inequality \eqref{Y14}. Thus it suffices to control the first three
terms of $M(t)$. Let $\varphi(x)$ be a smooth cut-off function with
$\varphi(x)=0$ for $|x|\leq 1/2$ and $\varphi(x)=1$ for
$|x|\geq2/3$. Let $v=\varphi(|x|/t) u$. Then for any $t\geq4$, by
Lemma \ref{lem6}, we have
\begin{align}\label{Y32}
&\sup_{4\leq s\leq t}\Big(\sum_{|a|\leq k-1}\||s-r|^{1/2}\partial\partial^av(s,\cdot)\|^2_{L^2(\mathbb R^2)}
\nonumber\\&\qquad\quad+\sum_{|a|\leq
k-1}\||s-r|^{-1/2}(\ln|s-r|)^{-1}\partial^av(s,\cdot)\|^2_{L^2(\mathbb
R^2)}\nonumber\\&\qquad\quad+\sum_{|a|\leq k-1}\|r^{1/2}(\partial_t+\partial_r)\partial^av(s,\cdot)\|^2_{L^2(\mathbb R^2)}\Big)\nonumber\\&\quad+\sum_{|a|\leq 
k-1}\int_4^t\int_{\mathbb
R^2}\frac{(s+r)\big((\partial_s+\partial_r)\partial^av\big)^2}{|s-r|(\ln|s-r|)^2}dxds
\nonumber\\&\leq C\varepsilon^2+C\sum_{|a|\leq
k-1}\int_4^t\int_{\mathbb R^2}|\Box_h\partial^av|
(|S\partial^av|+|\partial^av|)dxds \nonumber\\&\quad+C\sum_{|a|\leq
k-1}\int_4^t\int_{\mathbb R^2}|\partial h||\partial
\partial^av|\big(|S\partial^av|+|\partial^av|\big)dxds
\nonumber\\&\quad+C\sum_{|a|\leq k-1}\int_4^t\int_{\mathbb
R^2}|Sh||\partial \partial^av|^2dxds +C\sum_{|a|\leq
k-1}\int_4^t\int_{\mathbb
R^2}(\ln|s-r|)^{-2}|h||\partial\partial^av|^2dxds
\nonumber\\&\quad+C\sum_{|a|\leq k-1}\int_4^t\int_{\mathbb
R^2}|s-r|^{-1}(\ln|s-r|)^{-2}|h||\partial\partial^a
v|(|S\partial^av|+|\partial^av|)dxds
\nonumber\\&\quad+C\sum_{|a|\leq
k-1}\int_4^t\|r^{-1/2}|s-r|^{1/2}\partial\partial^av(s,\cdot)\|^2_{L^2(\mathbb
R^2)}ds.
\end{align}
A straight calculation shows that $\Box_h\partial^av=\varphi\Box_h\partial^au+R_0$, where $R_0$ are the terms with the form for $|b|+|c|\leq |a|$,
\begin{align*}
\partial\partial^b\varphi\partial\partial^cu,\ \partial^2\partial^b\varphi\partial^cu,\ h^{\alpha\beta}\partial\partial^b\varphi\partial\partial^cu,\ 
h^{\alpha\beta}\partial^2\partial^b\varphi\partial^cu.
\end{align*}
Since $|\partial\partial^b\varphi|\leq Ct^{-1}$ and $|\partial^2\partial^b\varphi|\leq Ct^{-2}$, we have
\begin{align*}
&\quad\int_4^t\int_{\mathbb R^2}|R_0||S\partial^av|dxds
\\&\leq C\sum_{|b|\leq k-1}\int_4^ts^{-1}\int_{\frac s2\leq |x|\leq \frac{2s}3}|\partial \partial^bu||S\partial^av|dxds
+C\int_4^ts^{-2}\int_{\frac s2\leq |x|\leq \frac{2s}3}|u||S\partial^av|dxds
\\&\leq C\sum_{|b|\leq k-1}\int_4^t\int_{\frac s2\leq |x|\leq \frac{2s}3}|\partial \partial^bu||\partial\partial^av|dxds
+C\int_4^ts^{-1}\int_{\frac s2\leq |x|\leq \frac{2s}3}|u||\partial\partial^av|dxds
\\&\leq C\sum_{|b|\leq k-1}\int_4^t\int_{\frac s2\leq |x|\leq \frac{2s}3}|\partial \partial^bu|^2dxds
+C\int_4^ts^{-2}\int_{\frac s2\leq |x|\leq \frac{2s}3}|u|^2dxds,
\end{align*}
which implies
\begin{align*}
&\quad\int_4^t\int_{\mathbb R^2}|R_0||S\partial^av|dxds
\\&\leq C\sum_{|b|\leq k-1}\int_4^ts^{-1}\int_{\mathbb R^2\backslash\mathcal{K}}|s-r||\partial \partial^bu|^2dxds
+C\int_4^ts^{-1}\int_{\mathbb R^2\backslash\mathcal{K}}\frac{|u|^2}rdxds.
\end{align*}
Thus
\begin{align}\label{Y33}
&\sum_{|a|\leq k-1}\int_4^t\int_{\mathbb R^2}|\Box_h\partial^av|
|S\partial^av|dxds\nonumber\\&\leq C\sum_{|a|\leq
k-1}\int_4^t\int_{\mathbb R^2}|\Box_h\partial^au|
|S\partial^av|dxds\nonumber\\&\quad+C\sum_{|a|\leq
k-1}\int_4^ts^{-1}\||s-r|^{1/2}\partial\partial^au(s,\cdot)\|_{L^2(\mathbb
R^2\backslash\mathcal {K})}^2ds
\nonumber\\&\quad+C\int_4^ts^{-1}\int_{\mathbb R^2\backslash\mathcal{K}}\frac{|u|^2}rdxds,
\end{align}
and
\begin{align}\label{Y36}
&\sum_{|a|\leq
k-1}\int_4^t\|r^{-1/2}|s-r|^{1/2}\partial\partial^av(s,\cdot)\|^2_{L^2(\mathbb
R^2)}ds \nonumber\\&\leq C\sum_{|a|\leq
k-1}\int_4^ts^{-1}\||s-r|^{1/2}\partial\partial^au(s,\cdot)\|_{L^2(\mathbb
R^2\backslash\mathcal {K})}^2ds
\nonumber\\&\quad+C\int_4^ts^{-1}\int_{\mathbb R^2\backslash\mathcal{K}}\frac{|u|^2}rdxds.
\end{align}
The decomposition $2Sv=(t+r)(v_t+v_r)+(t-r)(v_t-v_r)$ shows that
\begin{align}\label{Y34}
&\int_4^t\int_{\mathbb R^2}|\Box_h\partial^au||S\partial^av|dxds
\nonumber\\=&\int_4^t\int_{|x|\geq s/2}|\Box_h\partial^au||S\partial^av|dxds
\nonumber\\\leq& C\int_4^t\int_{|x|\geq s/2}|\Box_h\partial^au||r(\partial_s+\partial_r)\partial^av|dxds+C\int_4^t\int_{|x|\geq 
s/2}|\Box_h\partial^au||(s-r)\partial\partial^av|dxds
\nonumber\\\leq&
C\|r^{1/2}|s-r|^{1/2}(\ln|s-r|)\Box_h\partial^au\|^2_{L^2_sL^2_x([4,t]\times\{|x|\geq
s/2\})}
\nonumber\\&+C\||s-r|^{1/2}\Box_h\partial^au\|^2_{L_s^1L_x^2([4,t]\times\{|x|\geq
s/2\})}+R_1,
\end{align}
where we used H\"{o}lder inequality and Young inequality, and
\begin{align*}
R_1=&\delta \|r^{1/2}|s-r|^{-1/2}(\ln|s-r|)^{-1}(\partial_t+\partial_r)\partial^av\|^2_{L^2_sL^2_x([4,t]\times\{|x|\geq
s/2\})}
\\&+\delta\sup_{4\leq s\leq t}\||s-r|^{1/2}\partial\partial^av\|_{L_x^2(|x|\geq s/2)},
\end{align*}
which can be absorbed by  by the left-hand side of \eqref{Y32}, provided that $\delta$ is small enough.
\par Similarly, by H\"{o}lder inequality and Young inequality, we have
\begin{align}\label{Y42}
&\sum_{|a|\leq k-1}\int_4^t\int_{\mathbb R^2}|\Box_h\partial^av||\partial^av|dxds
\nonumber\\&\leq \sum_{|a|\leq k-1}\int_4^t\int_{\mathbb R^2}|\Box_h\partial^au||\partial^av|dxds+C\sum_{\substack{|a|\leq k-1\\|b|\leq k-1}}\int_4^t
s^{-1}\int_{\frac s2\leq |x|\leq \frac{2s}3}|\partial\partial^bu||\partial^av|dxds
\nonumber\\&\quad+C\sum_{\substack{|a|\leq k-1}}\int_4^t
s^{-2}\int_{\frac s2\leq |x|\leq \frac{2s}3}|u||\partial^av|dxds
\nonumber\\&\leq \sum_{|a|\leq k-1}\int_4^t\int_{\mathbb R^2}|\Box_h\partial^au||\partial^av|dxds+C\sum_{\substack{|a|\leq k-1\\|b|\leq k-1}}\int_4^t
\int_{\frac s2\leq |x|\leq \frac{2s}3}|\partial\partial^bu|^2dxds
\nonumber\\&\quad+C\sum_{\substack{|a|\leq k-1}}\int_4^t
s^{-2}\int_{\frac s2\leq |x|\leq \frac{2s}3}|u|^2dxds
\nonumber\\&\leq C\sum_{|a|\leq k-1}\||s-r|^{1/2}(\ln|s-r|)\Box_h\partial^au\|^2_{L_s^1L_x^2([4,t]\times\{|x|\geq s/2\})}
\nonumber\\&\quad+C\sum_{|a|\leq k-1}\int_4^ts^{-1}\||s-r|^{1/2}\partial\partial^au(s,\cdot)\|_{L^2(\mathbb R^2\backslash\mathcal {K})}^2ds+R_2
\nonumber\\&\quad+C\int_4^ts^{-1}\int_{\mathbb R^2\backslash\mathcal{K}}\frac{|u|^2}rdxds,
\end{align}
where
\begin{align*}
R_2=\delta\sup_{4\leq s\leq t}\sum_{|a|\leq k-1}\||s-r|^{-1/2}(\ln|s-r|)^{-1}\partial^av\|_{L_x^2(\mathbb R^2)}^2.
\end{align*}
Also the terms $R_2$ can be absorbed by the left-hand side of \eqref{Y32} provided that $\delta$ is small. We can apply similar argument to the remainder terms 
on the right-hand side of inequality \eqref{Y32}. This completes the proof of Theorem 4.1.
\end{proof}
\par In the remainder of this section, let
$h^{\alpha\beta}=-B_{\gamma}^{\alpha\beta}\partial_{\gamma}u$. We
shall prove that, for the local smooth solution to system
\eqref{Y1}, there exist positive constants $\varepsilon_0$ and $A$
such that
\begin{align}\label{Y26}
M(t)\leq 4C_0\varepsilon^2\ln T_{\varepsilon},
\end{align}
for all $4\leq t\leq T_{\varepsilon}$, where $\varepsilon\leq
\varepsilon_0$ and $\varepsilon^2
T_{\varepsilon}\ln^3T_{\varepsilon}=A$. By the local existence
theory, estimate \eqref{Y26} implies Theorem 1.1 holds. By the
continuity argument, to prove estimate \eqref{Y26}, it suffices to
prove estimate \eqref{Y26} holds with $4C_0$ replaced by $2C_0$,
assuming that inequality \eqref{Y26} holds.
\par For all $4\leq t\leq T_{\varepsilon}$, by Lemma \ref{lem4}, we have
\begin{align*}
(t+|x|)(|h|+|\partial h|)\leq Ct^{1/2}(t^{1/2}|h|+t^{1/2}|\partial
h|)\leq CT^{1/2}_{\varepsilon}M^{1/2}(t)\leq CA^{1/2}.
\end{align*}
Hence there exists a positive constant $A_0$ independent of
$\varepsilon$ such that inequalities \eqref{Y41} and \eqref{Y47}
hold, provided that $A\leq A_0$. Hence Theorem 4.1 holds. In the following we are going to
estimate the right-hand side of inequality \eqref{Y14}.
\subsection{Estimates of terms $I$ and $II$}
~\\
Since $[\partial,\Box]=0$ and $[\Omega,\Box]=0$, hence
\begin{align*}
\Box\partial^a\Omega^ju=\partial^a\Omega^j\big(B(\partial u)+B_{\gamma}^{\alpha\beta}\partial_{\gamma}u\partial_{\alpha}\partial_{\beta}u\big),
\end{align*}
which implies
\begin{align}\label{Y48}
&\sum_{\substack{|a|+j\leq
k-1\\j\leq1}}\|\Box_h\partial^a\Omega^ju(s,\cdot)\|_{L^2(\mathbb
R^2\backslash\mathcal {K})}\nonumber\\&\leq
C\sum_{\substack{|b|+|c|+j\leq
k-1\\j\leq1}}\|\partial\partial^bu\partial\partial^c\Omega^ju\|_{L^2(\mathbb
R^2\backslash\mathcal
{K})}\nonumber\\&\quad+C\sum_{\substack{|b|+|c|+j\leq k-1\\|c|+j\leq
k-2,j\leq1}}\|\partial\partial^bu\partial\nabla\partial^c\Omega^ju\|_{L^2(\mathbb
R^2\backslash\mathcal
{K})}\nonumber\\&\quad+C\sum_{\substack{|b|+|c|+j\leq k-1\\|b|\leq
k-2,j\leq1}}\|\partial\partial^c\Omega^ju\partial\nabla\partial^bu\|_{L^2(\mathbb
R^2\backslash\mathcal {K})},
\end{align}
by using Remark 1.3, we divide the estimate into two cases: $j=0$ or $j=1$. In the region $|x|\leq2s/3$, by using the
Morawetz energy norms and inequality \eqref{Y30}, it is not difficult
to show that the right-hand side of inequality \eqref{Y48} is
bounded by $Cs^{-1/2}M(s)$. In the region $|x|\geq2s/3$, for the case $j=0$, obviously the right-hand side of
inequality \eqref{Y48} is bounded by $Cs^{-1/2}M(s)$. When $j=1$ and
$|b|\leq|c|$, by inequality \eqref{Y25}, the right-hand side of
inequality \eqref{Y48} is bounded by $Cs^{-1/2}M(s)$. In the region
$|x|\geq2s/3$, when $j=1$ and $|b|\geq|c|$, by Lemma \ref{lem2} and Poincar\'{e}
inequality, the first two terms on the right-hand side of
\eqref{Y48} can be estimated as follows:
\begin{align*}
&\|\partial\partial^bu\partial\partial^c\Omega^ju\|_{L^2(|x|\geq2s/3)}
+\|\partial\partial^bu\partial\nabla\partial^c\Omega^ju\|_{L^2(|x|\geq2s/3)}
\\\leq &Cs^{-1/2}\|r^{1/2}\partial\partial^bv\partial\partial^c\Omega^jv\|_{L^2(\mathbb
R^2)}+Cs^{-1/2}\|r^{1/2}\partial\partial^bv\partial\nabla\partial^c\Omega^jv\|_{L^2(\mathbb
R^2)}
\\\leq& Cs^{-1/2}M(s),
\end{align*}
where $v=\varphi(|x|/t) u$ and $\varphi(x)$ is a smooth cut-off
function with $\varphi(x)=0$ when $|x|\leq1/2$ and $\varphi(x)=1$
when $|x|\geq2/3$. Since $\partial_1u=x_1/r\partial_ru-x_2/r^2\Omega
u$ and $\partial_1u=x_2/r\partial_ru+x_1/r^2\Omega u$, by Lemma \ref{lem3}
and Poincar\'{e} inequality, for $|b|+|c|\leq k-2$, $j=1$ and
$|b|\geq|c|$, in the region $|x|\geq2s/3$, the
last term on the right-hand side of inequality \eqref{Y48} can be
estimated as
\begin{align*}
&\|\partial\partial^c\Omega^ju\partial\nabla\partial^bu\|_{L^2(|x|\geq2s/3)}
\\\leq& Cs^{-1}M(s)+Cs^{-1/2}\|r^{1/2}\partial\partial^c\Omega^jv\partial_r\partial\partial^bv\|_{L^2(\mathbb R^2)}
\\\leq& Cs^{-1}M(s)+Cs^{-1/2}\|\partial\partial^c\Omega^jv\|_{H^1(\mathbb R^2)}\big(\|\partial_r\partial\partial^bv\|_{H^2(\mathbb R^2)}
+\|\Omega\partial\partial^bv\|_{L^2(\mathbb R^2)}\big)
\\\leq& Cs^{-1/2}M(s).
\end{align*}
Combining the discussion above together, we have
\begin{align*}
\sum_{\substack{|a|+j\leq
k-1\\j\leq1}}\|\Box_h\partial^a\Omega^ju(s,\cdot)\|_{L^2(\mathbb
R^2\backslash\mathcal {K})}\leq Cs^{-1/2}M(s).
\end{align*}
By Lemma \ref{lem4}, we have $|\partial h|\leq Cs^{-1/2}M^{1/2}(s)$, which
implies for any $4\leq t\leq T_{\varepsilon}$,
\begin{align*}
I\leq C\varepsilon^4T_{\varepsilon}\ln^2T_{\varepsilon}+C\varepsilon^{3}T^{1/2}_{\varepsilon}\ln^{3/2}T_{\varepsilon}.
\end{align*}
By Lemma \ref{lem4} again, it is not difficult to obtain 
\[II\leq
C\varepsilon^{3}T^{1/2}_{\varepsilon}\ln^{3/2}T_{\varepsilon}+C\varepsilon^4\ln^2T_{\varepsilon}.
\]
\subsection{Estimates of terms $III$ and $IV$}
~\\
For $|a|\leq k-2$ and $4\leq t\leq T_{\varepsilon}$, the term
$\|\langle t-r\rangle^{1/2}\partial^a\Box
u(t,\cdot)\|^2_{L^2(\mathbb R^2\backslash\mathcal {K})}$ is
controlled by $C\varepsilon^4\ln^2T_{\varepsilon}$. A straight
calculation shows that
\begin{align*}
|\Box_h\partial^au|\leq
C\sum_{|b|+|c|=|a|}|\partial\partial^bu||\partial\partial^cu|+C\sum_{\substack{|b|+|c|=|a|\\|c|\neq
|a|}}|\partial\partial^bu||\partial^2\partial^cu|.
\end{align*}
Thus by inequality \eqref{Y25} and using the scaling Morawetz energy norms, we
have
\begin{align*}
&\sum_{|a|\leq
k-1}\|r^{1/2}|s-r|^{1/2}(\ln|s-r|)\Box_h\partial^au\|^2_{L^2_sL^2_x([4,t]\times\{|x|\geq
s/2\})}\leq C\varepsilon^4T_{\varepsilon}\ln^4T_{\varepsilon},
\\&\sum_{|a|\leq
k-1}\||s-r|^{1/2}(\ln|s-r|)\Box_h\partial^au\|^2_{L_s^1L_x^2([4,t]\times\{|x|\geq
s/2\})}\leq C\varepsilon^4T_{\varepsilon}\ln^4T_{\varepsilon}.
\end{align*}
Thus for any $4\leq t\leq T_{\varepsilon}$, we have 
\[III\leq
C\varepsilon^4T_{\varepsilon}\ln^4T_{\varepsilon}.
\]
\par Next let us estimate $IV$.
Since $\|h(s,\cdot)\|_{L^{\infty}}\leq Cs^{-1/2}M^{1/2}(s)$ and
$\ln|s-r|\geq\ln2$ in the region $r\leq s-2$, we then have
\begin{align*}
&\sum_{|a|\leq k-1}\int_4^t\int_{\mathbb R^2\backslash\mathcal
{K}}(\ln|s-r|)^{-2}|h||\partial\partial^au|^2dxds
\\&\leq C\int_4^ts^{-1/2}M^{3/2}(s)ds
\\&\leq CT_{\varepsilon}^{1/2}\varepsilon^3\ln^{3/2}T_{\varepsilon}.
\end{align*}
By H$\ddot{o}$lder inequality, we have
\begin{align*}
&\sum_{|a|\leq
k-1}\int_4^t\int_{|x|\geq2s/3}|s-r|^{-1}(\ln|s-r|)^{-2}|h||\partial\partial^a
u||\partial^au|dxds
\\&\leq C\int_4^t\|h\|_{L^{\infty}(\mathbb
R^2\backslash\mathcal
{K})}M(s)ds
\\&\leq CT_{\varepsilon}^{1/2}\varepsilon^3\ln^{3/2}T_{\varepsilon}.
\end{align*}
By H$\ddot{o}$lder inequality and Poincar\'{e} inequality, we have
\begin{align*}
&\sum_{|a|\leq
k-1}\int_4^t\int_{|x|\leq2s/3}|s-r|^{-1}(\ln|s-r|)^{-2}|h||\partial\partial^a
u||\partial^au|dxds
\\&\leq CT_{\varepsilon}^{1/2}\varepsilon^3\ln^{3/2}T_{\varepsilon}.
\end{align*}
Then it remains to estimate
\begin{align*}
\sum_{|a|\leq k-1}\int_4^t\int_{\mathbb R^2\backslash\mathcal
{K}}|s-r|^{-1}(\ln|s-r|)^{-2}|h||\partial\partial^a
u||S\partial^au|dxds.
\end{align*}
It is easy to see that the above integral in the region $|x|\leq2s/3$ can be controlled by $CT_{\varepsilon}^{1/2}\varepsilon^3\ln^{3/2}T_{\varepsilon}$.
For this integral in the region $|x|\geq2s/3$, by the decomposition $2Su=(t+r)(u_t+u_r)+(t-r)(u_t-u_r)$ and H$\ddot{o}$lder inequality, we have
\begin{align*}
&\sum_{|a|\leq
k-1}\int_4^t\int_{|x|\geq2s/3}|s-r|^{-1}(\ln|s-r|)^{-2}|h||\partial\partial^a
u||S\partial^au|dxds
\\&\leq C\sum_{|a|\leq
k-1}\int_4^t\int_{|x|\geq2s/3}|s-r|^{-1}(\ln|s-r|)^{-2}|h||\partial\partial^a
u||(s+r)(\partial_s+\partial_r)\partial^au|dxds
\\&\quad+CT_{\varepsilon}^{1/2}\varepsilon^3\ln^{3/2}T_{\varepsilon}
\\&\leq C\sum_{|a|\leq k-1}\|r^{1/2}h\partial\partial^a
u\|_{L_s^2L_x^2([4,t]\times\{|x|\geq2s/3\})}M^{1/2}(t)+CT_{\varepsilon}^{1/2}\varepsilon^3\ln^{3/2}T_{\varepsilon}
\\&\leq CT_{\varepsilon}^{1/2}\varepsilon^3\ln^{3/2}T_{\varepsilon}.
\end{align*}
Hence we obtain 
\[IV\leq
CT_{\varepsilon}^{1/2}\varepsilon^3\ln^{3/2}T_{\varepsilon}.
\]
\subsection{Estimate of Term $V$}
~\\
Obviously for any $4\leq t\leq T_{\varepsilon}$, the first term in
$V$ is controlled by $C\varepsilon^4\ln^2 T_{\varepsilon}$.
\par For the second term in $V$, we first consider it inside the cone. A straight calculation yields that
\begin{align*}
&\int_{|x|\leq2s/3}|\Box_h\partial^au(s,x)||S\partial^au(s,x)|dx
\\\leq& C\sum_{|b|+|c|\leq |a|,|c|\neq |a|}\int_{|x|\leq2s/3}|\partial\partial^bu(s,x)||\partial^2\partial^cu(s,x)||S\partial^au(s,x)|dx
\\\leq& C\sum_{|b|+|c|\leq |a|,|c|\neq |a|}\int_{|x|\leq2s/3}|s-r||\partial\partial^bu(s,x)||\partial^2\partial^cu(s,x)||\partial\partial^au(s,x)|dx
\\\leq& Cs^{-1/2}\sum_{|b|+|c|\leq |a|,|c|\neq |a|}\int_{|x|\leq2s/3}(|s-r|^{1/2}|\partial\partial^bu(s,x)|)(|s-r|^{1/2}|\partial^2\partial^cu(s,x)|)
\\&\qquad\qquad\qquad\qquad\qquad\qquad\qquad\times(|s-r|^{1/2}|\partial\partial^au(s,x)|)dx
\\\leq& Cs^{-1/2}M^{3/2}(s),
\end{align*}
where we used H\"{o}lder inequality and inequality \eqref{Y30} in the last step.
\par
By H$\ddot{o}$lder inequality and inequality \eqref{Y30}, we have
\begin{align*}
&\int_{|x|\leq2s/3}|\Box_h\partial^au(s,x)||\partial^au(s,x)|dx\\\leq&
\|\Box_h\partial^au(s,\cdot)\|_{L^2(|x|\leq2s/3)}\|\partial^au(s,\cdot)\|_{L^2(|x|\leq2s/3)}
\\\leq& C\sum_{|b|+|c|\leq 
|a|,|c|\neq|a|}\|\partial\partial^bu(s,\cdot)\partial^2\partial^cu(s,\cdot)\|_{L^2(|x|\leq2s/3)}\|\partial^au(s,\cdot)\|_{L^2(|x|\leq2s/3)}
\\\leq& C(1+s)^{-1/2}\sum_{|b|+|c|\leq |a|,|c|\neq|a|}\|(|s-r|^{1/2}\partial\partial^bu)(|s-r|^{1/2}\partial^2\partial^cu)\|_{L^2(|x|\leq2s/3)}
\\&\times\Big(\sum_{|b|\leq |a|-1}\|\partial\partial^bu\|_{L^2(|x|\leq2s/3)}+\Big\|\frac u {r^{1/2}}\Big\|_{L^2(\mathbb R^2\backslash\mathcal{K})}\Big)
\\\leq& Cs^{-1/2}M^{3/2}(s).
\end{align*}
\par For the second term of $V$ close to the cone, by using the decomposition $2S=(t-r)(\partial_t-\partial_r)+(t+r)(\partial_t+\partial_r)$, inequality 
\eqref{Y25} and H$\ddot{o}$lder inequality, for any $4\leq t\leq T_{\varepsilon}$, we have
\begin{align*}
&\int_4^t\int_{|x|\geq2s/3}|\Box_h\partial^au(s,x)||S\partial^au(s,x)|dxds
\\\leq& C\int_4^t\int_{|x|\geq2s/3}|s-r||\Box_h\partial^au(s,x)||\partial\partial^au(s,x)|dxds
\\&+C\int_4^t\int_{|x|\geq2s/3}r|\Box_h\partial^au(s,x)||(\partial_t+\partial_r)\partial^au(s,x)|dxds
\\\leq& C\int_4^ts^{-1/2}M^{3/2}(s)ds+CM^{1/2}(t)\|r^{1/2}|s-r|^{1/2}\ln|s-r|\Box_h\partial^au\|_{L_s^2L_x^2([4,t]\times\{|x|\geq2s/3\})}
\\\leq& Ct^{1/2}\sup_{4\leq s\leq t}M^{3/2}(s)+CM^{1/2}(t)\ln t\sum_{|b|+|c|\leq |a|,|c|\neq 
|a|}\|r^{1/2}|s-r|^{1/2}\partial\partial^bu\partial^2\partial^cu\|_{L_s^2L_x^2([4,t]\times\{|x|\geq2s/3\})}
\\\leq& Ct^{1/2}\sup_{4\leq s\leq t}M^{3/2}(s)+CM^{1/2}(t)(\ln t) t^{1/2}\sup_{4\leq s\leq t}M(s)
\\\leq& CT_{\varepsilon}^{1/2}\varepsilon^3\ln^{5/2}T_{\varepsilon}.
\end{align*}
By H$\ddot{o}$lder inequality, for any $4\leq t\leq T_{\varepsilon}$, it holds that
\begin{align*}
&\int_4^t\int_{|x|\geq2s/3}|\Box_h\partial^au(s,x)||\partial^au(s,x)|dxds
\\\leq& C\Big(\sup_{0\leq s\leq t}M^{1/2}(s)\Big)\||s-r|^{1/2}\ln|s-r|\Box_h\partial^au\|_{L^1_sL_x^2([4,t]\times\{|x|\geq 2s/3\})}
\\\leq& C\Big(\sup_{0\leq s\leq t}M^{1/2}(s)\Big)\ln t\sum_{|b|+|c|\leq 
|a|,|c|\neq|a|}\int_4^ts^{-1/2}\|r^{1/2}|s-r|^{1/2}\partial\partial^bu\partial^2\partial^cu\|_{L^2(|x|\geq 2s/3)}dxds
\\\leq& C\Big(\sup_{0\leq s\leq t}M^{3/2}(s)\Big)t^{1/2}\ln t
\\\leq& CT_{\varepsilon}^{1/2}\varepsilon^3\ln^{5/2}T_{\varepsilon}.
\end{align*}
Thus for any $4\leq t\leq T_{\varepsilon}$, we have
\begin{align}\label{Y49}
&\sum_{|a|\leq k-1}\int_4^t\int_{\mathbb R^2\backslash\mathcal
{K}}|\Box_h\partial^au(s,x)|\big(|S\partial^au(s,x)|+|\partial^au(s,x)|\big)dxds
\nonumber\\&\leq
CT_{\varepsilon}^{1/2}\varepsilon^3\ln^{5/2}T_{\varepsilon}.
\end{align}
Similarly, for any $4\leq s\leq T_{\varepsilon}$, it holds
\begin{align*}
\sum_{j\leq k-1}\int_4^s\int_{\mathbb R^2\backslash\mathcal
{K}}|\Box_h\partial_{\tau}^ju(\tau,x)|\big(|S\partial_{\tau}^ju(\tau,x)|+|\partial_{\tau}^ju(\tau,x)|\big)dxd\tau
\leq Cs^{1/2}\varepsilon^3\ln^{5/2}T_{\varepsilon},
\end{align*}
which implies the third terms of $V$
\begin{align}\label{Y50}
&\sum_{j\leq k-1}\int_4^ts^{-1}\int_4^s\int_{\mathbb
R^2\backslash\mathcal
{K}}|\Box\partial_{\tau}^ju(\tau,x)|\big(|S\partial_{\tau}^ju(\tau,x)|+|\partial_{\tau}^ju(\tau,x)|\big)dxd\tau
ds
\nonumber\\&\leq CT_{\varepsilon}^{1/2}\varepsilon^3\ln^{5/2}T_{\varepsilon}.
\end{align}
\par A similar argument to the estimates \eqref{Y49} and \eqref{Y50} shows that,
for any $4\leq t\leq T_{\varepsilon}$, the fourth and fifth terms in $V$
can be controlled by $CT_{\varepsilon}^{1/2}\varepsilon^3\ln^{5/2}T_{\varepsilon}$.
\par In the region $|x|\leq3s/4$, the sixth term in $V$ is controlled by $CT_{\varepsilon}^{1/2}\varepsilon^3\ln^{3/2}T_{\varepsilon}$. For the sixth term
of $V$ in the region $|x|\geq3s/4$, we have
\begin{align*}
&\int_4^t\int_{|x|\geq3s/4}|Sh||\partial\partial^au(s,x)|^2dxds
\nonumber\\\leq& C\int_4^t\int_{|x|\geq3s/4}|S\partial u||\partial\partial^au(s,x)|^2dxds
\nonumber\\\leq& V_1+V_2,
\end{align*}
where
\begin{align*}
&V_1=C\int_4^t\int_{|x|\geq3s/4}r|(\partial_s+\partial_r)\partial u||\partial\partial^au(s,x)|^2dxds,
\nonumber\\&V_2=C\int_4^t\int_{|x|\geq3s/4}|s-r||(\partial_s-\partial_r)\partial u||\partial\partial^au(s,x)|^2dxds.
\end{align*}
A straight calculation shows that
\begin{align*}
V_2&\leq C\int_4^t\int_{|x|\geq3s/4}|s-r||\partial^2u||\partial\partial^au(s,x)|^2dxds
\\&\leq CM(t)\int_4^ts^{-1/2}\|r^{1/2}\partial^2u\|_{L^{\infty}(|x|\geq 3s/4)}ds
\\&\leq CM^{3/2}(t)t^{1/2}\leq CT_{\varepsilon}^{1/2}\varepsilon^3\ln^{3/2}T_{\varepsilon}.
\end{align*}
We divide the estimate of $V_1$ into two cases: $|a|\geq k-2$ or $|a|\leq k-3$. When $|a|\leq k-3$, by H\"{o}lder inequality,
\begin{align*}
V_1&\leq C\int_4^t\int_{|x|\geq3s/4}|r^{1/2}|s-r|^{-1/2}(\partial_s+\partial_r)\partial 
u|||s-r|^{1/2}\partial\partial^au(s,x)||r^{1/2}\partial\partial^au(s,x)|dxds
\\&\leq C\int_4^t\|r^{1/2}|s-r|^{-1/2}(\partial_s+\partial_r)\partial u\|_{L_x^2(|x|\geq 
3s/4)}\cdot\||s-r|^{1/2}\partial\partial^au\|_{L_x^2}\cdot\|r^{1/2}\partial\partial^au\|_{L^{\infty}_x}ds
\\&\leq CM^{3/2}(t)t^{1/2}\ln t\leq CT_{\varepsilon}^{1/2}\varepsilon^3\ln^{5/2}T_{\varepsilon}.
\end{align*}
When $|a|\geq k-2$, by H\"{o}lder inequality and using the decomposition $\partial_s=(\partial_s+\partial_r)-\partial_r$, $\nabla=\frac xr\partial_r-\frac 
x{r^2}\wedge\Omega$, we have
\begin{align}\label{Y57}
V_1 &\leq C\int_4^t\||r(\partial_s+\partial_r)\partial u|\cdot|\partial_s\partial^au(s,x)|^2\|_{L_x^1(|x|\geq3s/4)}ds
\nonumber\\&\quad+ C\int_4^t\||r(\partial_s+\partial_r)\partial u|\cdot|\nabla\partial^au(s,x)|^2\|_{L_x^1(|x|\geq3s/4)}ds
\nonumber\\&\leq C\int_4^t\||r(\partial_s+\partial_r)\partial u|\cdot|(\partial_s+\partial_r)\partial^au(s,x)|^2\|_{L_x^1(|x|\geq3s/4)}ds
\nonumber\\&\quad+ C\int_4^ts^{-1}\||(\partial_s+\partial_r)\partial u|\cdot|\Omega\partial^au(s,x)|^2\|_{L_x^1(|x|\geq3s/4)}ds
\nonumber\\&\quad+C\int_4^t\||r(\partial_s+\partial_r)\partial u|\cdot|\partial_r\partial^au(s,x)|^2\|_{L_x^1(|x|\geq3s/4)}ds
\nonumber\\&\leq C\int_4^ts^{-1/2}\|r^{1/2}(\partial_s+\partial_r)\partial 
u\|_{L_x^{\infty}}\cdot\|r^{1/2}(\partial_s+\partial_r)\partial^au(s,x)\|^2_{L_x^2(|x|\geq3s/4)}ds
\nonumber\\&\quad+ C\int_4^ts^{-3/2}\|r^{1/2}(\partial_s+\partial_r)\partial 
u\|_{L_x^{\infty}}\cdot\Big(\sum_{|b|\leq|a|-1}\|\Omega\partial\partial^bu(s,x)\|^2_{L_x^2}+s\Big\|\frac{\Omega u}{r^{1/2}}\Big\|^2_{L^2}\Big)ds
\nonumber\\&\quad+CM^{1/2}(t)\int_4^t\|||s-r|^{-1/2}r(\partial_s+\partial_r)\partial u|\cdot|\partial_r\partial^au(s,x)|\|_{L_x^2(|x|\geq3s/4)}ds
\nonumber\\&\leq Ct^{1/2}M^{3/2}(t)+CM^{1/2}(t)\int_4^t\|r^{1/2}v\partial_r\partial^aw\|_{L^2(\mathbb
R^2)}ds,
\end{align}
where $w(s,x)=\zeta(|x|/s)u(s,x)$,
$v(s,x)=r^{1/2}(\partial_s+\partial_r)\partial w(s,x)|s-r|^{-1/2}$,
and $\zeta(x)$ is a smooth cut-off function with $\zeta(x)=0$ when
$|x|\leq2/3$ and $\zeta(x)=1$ when $|x|\geq3/4$.
Thus by Lemma \ref{lem3}, for any $4\leq t\leq T_{\varepsilon}$, the second term on the right-hand side of \eqref{Y57} is controlled by 
$CM(t)\int_4^t\|v\|_{H^1(\mathbb R^2)}ds$. Since
\begin{align*}
\|v\|_{H^1(\mathbb R^2)}&\leq Cs^{-1}\sum_{1\leq|a|\leq 2}\|\partial^a u\|_{L^2(\mathbb R^2\backslash\mathcal {K})}+C\|r^{-1/2}|s-r|^{-1/2}\partial^2
u\|_{L^2(|x|\geq2s/3)}\nonumber\\&\quad+C\sum_{|a|\leq2}\|r^{1/2}
|s-r|^{-1/2}(\partial_s+\partial_r)\partial^au\|_{L^2(|x|\geq2s/3)},
\end{align*}
then \eqref{Y57} can be controlled by
$CT_{\varepsilon}^{1/2}\varepsilon^3\ln^{5/2}T_{\varepsilon}$.
Similarly the last term in $V$ can also be controlled by
$CT_{\varepsilon}^{1/2}\varepsilon^3\ln^{5/2}T_{\varepsilon}$.
Hence it holds 
\begin{align*}
V\leq
CT_{\varepsilon}^{1/2}\varepsilon^3\ln^{5/2}T_{\varepsilon}+C\varepsilon^4\ln^2
T_{\varepsilon}.
\end{align*}
\par In conclusion, by combining the discussion above together, we get
\begin{align}\label{Y51}
M(t)\leq C_0\varepsilon^2\ln T_{\varepsilon}+C\varepsilon^4T_{\varepsilon}\ln^2T_{\varepsilon}
+C\varepsilon^4T_{\varepsilon}\ln^4T_{\varepsilon}+CT_{\varepsilon}^{1/2}\varepsilon^3\ln^{5/2}T_{\varepsilon}.
\end{align}
Thus there exists a sufficiently small positive constant $A$ such
that $M(t)\leq 2C_0\varepsilon^2\ln T_{\varepsilon}$, provided that
$\varepsilon^2T_{\varepsilon}\ln^3T_{\varepsilon}=A$ and
$\varepsilon$ is small enough. This completes the proof of Theorem \ref{thm1}.
\section*{Acknowledgement}

The first author would like to express his sincere thank to Professor Jason Metcalfe for his helpful discussion and suggestions. All the authors are grateful to Professor Yi Zhou for his helpful discussion.

Ning-An Lai is partially supported by NSFC (No. 12271487, 12171097, W2521007).
Wei Xu is partially supported by the Doctoral Initiation Fund of Nanchang Hangkong University (No. EA202207232).


\end{document}